\documentclass[11pt]{amsart}
\usepackage[left=3.5cm, right=2.5cm, top=3cm, bottom=3cm]{geometry}

\usepackage{amsmath,amsthm,amsfonts,amssymb}

\newcommand*{\E}{\mathbb{E}}
\newcommand*{\Prob}{\mathbb{P}}

\newcommand*{\R}{\mathbb{R}}

\newcommand*{\Ind}{\mathbb{I}}

\newcommand*{\eps}{\varepsilon}
\newcommand*{\Id}{\mathbb{I}}
\newcommand{\setDiff}{\vartriangle}

\renewcommand{\phi}{\varphi}

\usepackage{ifthen}
\newcommand*{\equ}[2][\empty]{
	\begin{equation}
	\ifthenelse{\equal{#1}{\empty}}{}{\label{#1}}
	#2
	\end{equation}
}

\newcommand*{\al}[1]{\begin{align*}#1\end{align*}}

\def\wt{\widetilde}
\def\E{\mathbb E}
\def\R{\mathbb R}
\def\S{\mathbb S}
\def\Pb{\mathbb P}
\def\Z{\mathbb Z}
\def\Sc{\mathcal S}

\def\Mc{\mathcal M}
\def\Dc{\mathcal D}
\def\Ec{\mathcal E}
\def\Bc{\mathcal B}
\def\Ic{\mathcal I}
\def\Fc{\mathcal F}
\def\Hc{\mathcal H}
\def\Kc{\mathcal K}
\def\ind{{\mathbb I}}

\newtheorem{theorem}{Theorem}
\newtheorem*{theorem*}{Theorem}

\newtheorem{lemma}{Lemma}
\newtheorem{proposition}{Proposition}

\theoremstyle{definition}

\begin{document}

\title[Gaussian processes, bridges and membranes]{Gaussian processes, bridges and membranes \\extracted from selfsimilar random fields}

\keywords{self-similarity, generalized random field, fractional Brownian motion, Gaussian membrane}
\subjclass{60G15, 60G60}

\author{Maik G\"orgens}
\author{Ingemar Kaj}
\address{Department of Mathematics, Uppsala University}
\curraddr{P.O. Box 480, 751 06 Uppsala, Sweden}
\email{maik@math.uu.se, ikaj@math.uu.se}

\date{\today}

\begin{abstract}
We consider the class of selfsimilar Gaussian generalized random
fields introduced in Dobrushin \cite{D}. These fields are indexed by
Schwartz functions on $\R^d$ and parametrized by a self-similarity
index and the degree of stationarity of their increments.  We show
that such Gaussian fields arise in explicit form by letting Gaussian
white noise, or Gaussian random balls white noise, drive a shift and
scale shot-noise mechanism on $\R^d$, covering both isotropic and
anisotropic situations.  In some cases these fields allow indexing
with a wider class of signed measures, and by using families of signed
measures parametrized by the points in euclidean space we are able to
extract pointwise defined Gaussian processes, such as fractional
Brownian motion on $\R^d$.  Developing this method further, we construct
Gaussian bridges and Gaussian membranes on a finite domain, which
vanish on the boundary of the domain. 
\end{abstract}

\maketitle

%%%%%%%%%%%%%%%%%%%%%%%%%%%%%%%%%%%%%%%%
%%%%%%%%%%%%%%%%%%%%%%%%%%%%%%%%%%%%%%%%

\section{Introduction}

The main purpose of this work is to propose a method for constructing
a variety of Gaussian random processes on $\R^d$ by pointwise
evaluation of Gaussian selfsimilar random fields.  We will work with
zero mean Gaussian fields $X$ defined with respect to Schwartz functions $\Sc$
on $\R^d$ or, more generally, with respect to a class of signed
measures $\Mc$ on the Borel sets $\Bc(\R^d)$, writing $\phi\mapsto
X(\phi)$, $\phi\in\Sc$, and $\mu\mapsto X(\mu)$, $\mu\in \Mc$.
Defining the dilations $\phi_c$ of $\phi$ and $\mu_c$ of $\mu$, by
\[
\phi_c(x)=c^{-d}\phi(c^{-1}x),\quad  x\in \R^d, 
\qquad \mu_c(A)=\mu(c^{-1}A), \quad A\subset \Bc(\R^d), 
\]
a random field is said to be selfsimilar with self-similarity
index $H$, if 
\[
X(\phi_c)\stackrel{d}{=} c^{H} X(\phi), \quad 
X(\mu_c)\stackrel{d}{=} c^{H} X(\mu), \quad c>0. 
\]
Dobrushin \cite{D}, pioneered a theory of generalized random fields
with $r$th order stationary increments, and characterized all Gaussian
selfsimilar random fields on $\R^d$ by providing a representation of
the covariance functional $C(\phi,\psi)=\mathrm{Cov}(X(\phi),X(\psi))$
parametrized by $r$ and $H$.  We will use special instances of such
random fields $\mu\mapsto X(\mu)$ with $H>0$ and extract Gaussian
processes $(X_t)_{t\in \R^d}$ by putting $X_t=X(\mu_t)$ for a suitably
chosen family of indexing measures $(\mu_t)_{t\in \R^d}$.

Gaussian white noise $M_d(dx)$, which is the case $r=0$ and $H=-d/2$, is
such that $M_d(\phi)$ is a zero mean Gaussian random field with covariance
$C(\phi,\psi)=\int_{\R^d} \phi(x)\psi(x)\,dx$.  Gaussian random balls
white noise is a class of isotropic, generalized random fields
$W_\beta$, such that for a suitable family $\Mc_\beta$ of signed
measures,
\[
W_\beta(\mu)=\int_{\R^d\times \R_+} \mu(B(x,u))\,M_\beta(dx,du),\quad
\mu\in\Mc_\beta,
\]
where $B(x,u)$ is the Euclidean ball centered in $x$ with radius $u$
and $M_\beta(dx,du)$ is Gaussian white noise on $\R^d\times
[0,\infty)$ with control measure $\nu(dx,du)=dx\,u^{-\beta-1}du$. Such
  fields are known to be well-defined for $d-1<\beta<d$ and
  $d<\beta<2d$ and $W_\beta$ is selfsimilar with index
  $H=(d-\beta)/2\in (-d/2,0)\cap (0,1/2)$,
  \cite{BEK2010},\cite{KLNS2007}.  These classes of selfsimilar random
  fields may be recognized as the cases $r=0$, $-d/2<H<0$ and $r=1$,
  $0<H<1/2$, respectively, of isotropic fields in Dobrushin's
  characterization. By considering the Riesz transform 
\[
(-\Delta)^{-m/2}\phi(x)=\int_{\R^d} |x-y|^{-(d-m)}\phi(y)\,dy,\quad 0<m<d,
\]
and random fields defined by
\[
X(\phi)=W_{\beta}((-\Delta)^{-m/2}\phi),\quad 
\]
for a suitably restricted class of test functions $\phi$, it is
also possible to extend the range of the self-similarity index $H$
covered by random balls models to any value
$H\not=\Z$ if $d\ge 2$ and $H\not=\frac{1}{2}\Z$ if $d=1$, see \cite{BEK2010}. 

In this work we present a more general construction of Gaussian
selfsimilar shot noise random fields, which naturally includes
anisotropic models. We apply the same Gaussian white noises, $M_d$ and
$M_\beta$ as above, use the method of indexing random fields with a
class of signed measures, and extend the range of self-similarity index
with the help of the Riesz transform. These tools allow us to build,
in particular, Gaussian selfsimilar random fields $\mu\mapsto X(\mu)$
with index $H>0$, and apply to them a family of measures $(\mu_t)_{t \in \R^d}$.
By extracting the random fields in this manner, we obtain pointwise
defined random processes
\[
t\mapsto Y_t=X(\mu_t),\quad t\in \R^d,
\]
which inherit relevant properties from the underlying random fields.
The guiding example is fractional Brownian motion $B_H(t)$,
$t\in\R^d$, with $0<H<1$, which we extract from an appropriate random
field by applying $\mu_t=\delta_t-\delta_0$ and/or
$\mu_t=(-\Delta)^{-m/2}(\delta_t-\delta_0)$ with a suitable $m$.  As a
byproduct we obtain a new representation of fractional Brownian motion
in terms of $M_\beta$, which may be compared to the well-balanced
representation that results from using $M_d$. To illustrate isotropy
and anisotropy in natural situations, we also compare the random balls
construction with a random cylinder model, which leads to a comparison
between fractional Brownian motions and fractional Brownian sheets.  

To investigate further the range of applicability of the briefly
explained extraction principle, we consider for the one-dimensional
case $d=1$ construction of Gaussian bridges on an interval of the real line and
construction of Volterra processes. In higher dimensions we propose
the construction of membranes on a bounded domain $D$ in
$\R^d$, as Gaussian processes $X_t$, $t\in D$, such that $X_t$ converges in probability to $0$ as
$t$ tents to $\partial D$. Finally, we discuss membranes obtained from Gaussian
random balls white noise, which is thinned by a hard boundary in the
sense that balls that do not fall entirely within the domain are
discarded.

Our presentation is organized as follows. In the next Section 2 we
give preliminaries on Gaussian random measures and fields including an
account of Dobrushin's characterization of selfsimilar random fields.
In Section 3 we present our main results on Gaussian shot noise random
fields as Theorem \ref{thmshotnoise}, devoted to fields generated by a
wide range of pulse functions and random balls white noise $M_\beta$,
and Theorem \ref{thmsingularshot}, which instead applies a singular
shot function $h_\beta$ and regular white noise $M_d$. The discussion
on random cylinder models is included as a separate
subsection. Section 4 contains our account of the extraction method
and the various results on fractional Brownian motion, Gaussian
bridges, Volterra processes and membranes constructed by soft boundary
thinning of the harmonic measure. Finally, Section 5 is devoted to
membranes generated by hard boundary thinning.

%%%%%%%%%%%%%%%%%%%%%%%%%%%%%%%%%%%%%%%%

\section{Preliminaries on Gaussian random measures and fields}

Let $(D,\Dc,\nu)$ be a measure space and let $\Dc_\nu=\{A\in\Dc:
\nu(A)<\infty\}$ denote the set of measurable sets with finite
measures.  A Gaussian stochastic measure on $(D,\Dc,\nu)$ is a family
of centered Gaussian random variables $Z(A)$, $A\in\Dc_\nu$, such that
\[
\mathrm{Cov}(Z(A),Z(B))=\nu(A\cap B), \quad A,B\in \Dc_\nu,
\]
and the corresponding Gaussian stochastic integral $f\mapsto \int
f\,dZ$ is the linear isometry $f\mapsto \Ic(f)$ of $L^2(D,\Dc,\nu)$
into a Gaussian Hilbert space $H$, defined by $\Ic(\ind_A)=Z(A)$,
$A\in \Dc_\nu$, \cite{SJ} Ch.\ 7.2.  Our main examples will be the
Euclidean case $D=\R^d$ with control measure $\nu(dx)$ which is
uniform or absolutely continuous with respect to Lebesgue measure on
$\R^d$, and simple product spaces, such as $D=\R^d\times \R_+$
equipped with a product measure $\nu(dx,du)=dx\,\nu_\gamma(du)$, where
$\nu_\gamma(du)=u^{-\gamma-1}\,du$ is a power law measure on the real
positive line.

%%%%%%%%%%%%%%%%%%%%%%%%%%%%%%%%%%%%%%%%

\subsection*{Gaussian white noise on $\R^d$}

We denote by $M_d(dx)$ the Gaussian stochastic measure on
$(\R^d,\Bc(\R^d),dx)$, the $d$-dimensional Euclidean space with the
Borel $\sigma$-algebra $\Bc(\R^d)$ and Lebesgue control measure $dx$.  The
stochastic integral with respect to $M_d$ is the linear map $f\mapsto
\Ic(f)=\int_{\R^d} f(x)\,M_d(dx)$ defined as an isometry from
$L^2(\R^d,\Bc(\R^d),dx)$, equipped with the inner product norm
$\|\cdot\|=\sqrt{\langle \,,\,\rangle}$, where $\langle f,g\rangle=\int
fg\,dx$, into a Gaussian space $L^2(\Omega,\Fc,\Pb)$.  Let $\E$ be the
expectation operator associated with $\Pb$. Since
\[
\int_{\R^d} f(x)\,M_d(dx) \int_{\R^d}
g(y)\,M_d(dy)=\int_{\R^d}f(x)g(x)\,dx,
\] 
the covariance functional $\E(\Ic(f)\,\Ic(g))=\langle
f,g\rangle$, is given by the ordinary inner product of $L^2$
functions.  The same construction works in greater generality,
such as anisotropic white noise with control measure $w(x)\,dx$ for a
nonnegative weight function $w$ and covariance functional given by the
inner product of the weighted space $L_2(\R^d,w\,dx)$.

%%%%%%%%%%%%%%%%%%%%%%%%%%%%%%%%%%%%%%%%

\subsection*{Gaussian Hilbert space}

Let $\Sc$ be the space of real, rapidly decreasing and smooth Schwartz
functions on $\R^d$. The continuous, bilinear form $\langle
\,,\,\rangle$ is symmetric, semi-definite and non-degenerate on
$\Sc$. Hence $(\Sc,\langle \,,\,\rangle)$ is a pre-Hilbert space with
inner product $\langle f,g\rangle$ for which the completion to a
Hilbert space is the usual space $L^2(\R^d)$ of real-valued
square-integrable functions on $\R^d$.  Also, by Minlos's theorem,
$\langle f,g\rangle$ corresponds to a unique Gaussian measure $\Pb$ on
the space $\Sc'$ of real tempered distributions, the dual space of
$\Sc$. Indeed, we obtain a Gaussian Hilbert space $H\subset L^2(\Pb)$
such that the linear functional $f\mapsto u(f)$ on $\Sc'$ is an
isometry which defines the Gaussian white noise measure on $\Sc'$. As
a Gaussian field on an $L^2$-space, white noise on generalized
Schwartz distributions may be regarded as the stochastic integral
$f\mapsto \int f(x)\,M_d(dx)$, cf.\ \cite{SJ}, Ex.\ 1.16, Ex.\ 7.24.

%%%%%%%%%%%%%%%%%%%%%%%%%%%%%%%%%%%%%%%%

\subsection*{Stationary Gaussian random fields}

We write $|j|=\sum_{k=1}^d j_k$ for each $d$-dimensional multi-index $j=(j_1,\dots,j_d)$ and $x^j=\prod_{k=1}^d x_k^{j_k}$,
$x=(x_1,\dots,x_d)\in \R^d$, and consider the sequence $\Sc_r$,
$r=0,1,\dots$, of closed subspaces of $\Sc$, such that 
\[
\Sc_r=\Big\{\phi\in \Sc: \int_{\R^d} x^j\phi(x)\,dx=0,\, |j|<r\Big\},\quad
r=1,2\dots,\quad\Sc_0=\Sc.
\]
A Gaussian random field over $\Sc_r$ is a continuous, linear functional
$X: \Sc_r\to \R$, such that $X(\phi)$ is a Gaussian random variable
for each $\phi\in \Sc_r$.  The field is said to be isotropic if the
distribution is invariant under rotations of $\R^d$ and stationary if
it is invariant under translations.  A Gaussian random field $X$ over
$\Sc_0$ is said to have stationary $r$th increments if the restriction
of $X$ to $\Sc_r$ is a stationary Gaussian random field over $\Sc_r$. 
Let $\Ec$ be the symmetric semidefinite bilinear form on $\Sc_r$,
defined by $\Ec(\phi,\psi)=\E X(\phi)X(\psi)$. Then $(\Sc_r,\Ec)$ is a
pre-Hilbert space with inner product $\Ec(\phi,\psi)$, which may be
completed to a Hilbert space $\Sc_\Ec$ with norm
$\sqrt{\Ec(\phi,\phi)}$, and then $\phi\mapsto X(\phi)$ is an isometry
of $\Sc_\Ec$ onto a Gaussian Hilbert space in $\Sc_r'$.  Conversely, by
Minlos's theorem, any continuous bilinear semidefinite symmetric form
gives rise to a unique Gaussian field on $\Sc_r'$. 

More generally, we may consider Gaussian random fields defined on a
space of measures. Let $(\Mc,\|\cdot\|)$ denote the normed space of
signed measures $\mu$ on $\R^d$ with variation measure $|\mu|$, such
that the total variation norm is finite,
$\|\mu\|=|\mu|(\R^d)<\infty$. We put $\Mc_0=\Mc$ and for $r=1,2\dots$,
\begin{equation}\label{defspaceofmeasures}
\Mc_r=\left\{\mu\in \Mc: \int_ {\R^d} |x|^{r-1}\,|\mu|(dx)<\infty, 
\int_ {\R^d} x^j\,\mu(dx)=0,\; |j|<r\right\}.
\end{equation}
The subspaces $\Mc_r$ are closed under translations $\mu(A)\mapsto
\mu(A-s)$, $s\in\R^d$, $A\in\Bc(\R^d)$.  In this framework a Gaussian
random field $X$ over $\Mc _r$ is defined in analogy to those over
$\Sc_r$, and the notions of isotropy, translation invariance and $r$th
order stationary increments carry over.  Moreover, by completion one
obtains a Gaussian Hilbert space $\Mc_\Ec$ and an isometry $\mu\mapsto X(\mu)$
onto a Gaussian Hilbert space in the dual space of distributions, cf.\
\cite{SJ} Def.\ 1.18, and \cite{BEK2010} Sect.\ 3.1.  

%%%%%%%%%%%%%%%%%%%%%%%%%%%%%%%%%%%%%%%%

\subsection*{The M.\ Riesz potential kernel}

Let $\Delta=\partial^2/\partial x_1^2+\dots+\partial^2/\partial x_d^2$
be the usual Laplacian operator on $\R^d$. The Fourier transform
$\widehat{\Delta\phi}$, $\phi\in\Sc(\R^d)$, satisfies
\[
\widehat{\Delta\phi}(\xi)=-|\xi|^2\widehat\phi(\xi),\quad \xi\in \R^d.
\]
Then, for any $m\in \Z$, the power operators $(-\Delta)^{-m/2}$ of
the Laplace operator may be defined formally using the Fourier transform, by 
\begin{equation}\label{rieszfourier}
\widehat{(-\Delta)^{-m/2}\phi}(\xi)=|\xi|^{-m}\widehat\phi(\xi),\quad
\xi\in \R^d. 
\end{equation}
In the context of random fields the family of operators
$(-\Delta)^{-m/2}$, $m\in\Z$, can be given a precise meaning as linear
homeomorphisms defined on the intersection space
$\Sc_\infty=\cap_{r\ge 0} \Sc_r$, see \cite{BEK2010}. For $1\le m\le d-1$,
and more generally for a non-integer parameter $m$, $0<m<d$, the
application $(-\Delta)^{-m/2}\phi$ is well-defined for $\phi\in\Sc$
and can be realized as a fractional integral with respect to the Riesz
kernel, given by
\[
(-\Delta)^{-m/2}\phi(x)=C_{m,d}\int_{\R^d}
|x-y|^{-(d-m)}\phi(y)\,dy,\quad 
C_{m,d}=\frac{\Gamma((d-m)/2)}{\pi^{d/2}2^m\Gamma(m/2)}.
\]
In one dimension, $d=1$, this extends naturally by putting 
$(-\Delta)^{-1/2}\phi(x)=\int_x^\infty \phi(y)\,dy$.  

For signed measures in $\mu\in\Mc$ we will understand 
$(-\Delta)^{-m/2}\mu$ to be the map generated by the Riesz potential
of order $m$,  defined by  
\[
(-\Delta)^{-m/2}\mu(dx)=C_{m,d}\int_{\R^d}
|x-y|^{-(d-m)}\,\mu(dy)\,dx. 
\]
For the one-dimensional case, $(-\Delta)^{-1/2}\mu(dx)=\int_x^\infty
\,\mu(dy)\,dx$.  The Riesz potential of order $m$ is finite almost
everywhere if and only if \cite{landkof}
\[
\int_{\{y\in\R^d:\, |y|>1\}} \frac{\mu(dy)}{|y|^{d-m}}<\infty, 
\]
and this condition will be satisfied for all measures $\mu$ considered
here. With regards to the Riesz kernel we will make frequent use of the
composition rule
\begin{equation}\label{rieszcomp}
\int_{\R^d} \frac{C_{m_1,d}}{|y-x|^{d-m_1}}\frac{C_{m_2,d}}{|y'-x|^{d-m_2}}\,dx 
=\frac{\mathrm{C_{m_1+m_2,d}}}{|y-y'|^{d-m_1-m_2}},
\end{equation}
valid for $0<m_1,m_2<d$, $m_1+m_2<d$.

%%%%%%%%%%%%%%%%%%%%%%%%%%%%%%%%%%%%%%%%

\subsection*{Selfsimilar Gaussian random fields}

For $\phi\in \Sc$ let $\phi_c$ be the dilation defined by
$\phi_c(x)=c^{-d}\phi(c^{-1}x)$, $c\ge 0$.  Clearly, $\phi_c\in \Sc_r$
if $\phi\in \Sc_r$.  A random field $X$ over $\Sc_r$ is said to be
selfsimilar with index $H$, or $H$-selfsimilar, if $X(\phi_c)$ has the
same distribution as $c^H X(\phi)$, $\phi\in \Sc_r$. Similarly, for
$\mu\in \Mc(\R^d)$ define $\mu_c$ by $\mu_c(B)=\mu(B/c)$,
$B\in\Bc(\R^d)$.   We will sometimes write $\mu\mapsto X(\mu)$ for the
mapping of a random field even if the space of measures coincides with
the absolutely continuous signed measures $\mu(dx)=\phi(x)\,dx$,
$\phi\in \Sc_r$.  In this notation, a random field is $H$-selfsimilar
if $X(\mu_c)$ has the same distribution as $c^H\,X(\mu)$, for all
relevant $\mu$.

\begin{theorem}[Dobrushin '79 \cite{D}] \label{thmdobrushin}
Fix $r\ge 0$. A Gaussian random field
  $X$ on $\Sc_r$ is stationary and $H$-selfsimilar if and only if the
  covariance functional $C(\phi,\psi)={\rm Cov}(X(\phi),X(\psi))$ is
  given by 
\begin{align*}
C(\phi,\psi)&=\sum_{|j|=|k|=r} a_{jk}\int_{\R^d} x^j\phi(x)\,dx\int_{\R^d}
y^k\psi(y)\,dy\\
&\qquad \qquad +\int_{\S^{d-1}}\int_0^\infty \widehat \phi(u\theta)\overline{\widehat
  \psi(u\theta)} u^{-2H-1}\,du\,\sigma(d\theta),
\end{align*}
where the matrix $(a_{jk})$ is symmetric and nonnegative definite and
$\sigma(d\theta)$ is a finite, positive, and reflection-invariant
measure on the unit sphere $\S^{d-1}$ in $\R^d$.  Here, if $H>r$ then
$X\equiv 0$, if $H=r$ then $\sigma(d\theta)=0$ and if $H<r$ then
$(a_{jk})=0$.
\end{theorem}

%%%%%%%%%%%%%%%%%%%%%%%%%%%%%%%%%%%%%%%%

\subsection*{Random polynomials}  

The special case $H=r$ in Theorem \ref{thmdobrushin} corresponds to random
polynomials.  For $x\in \R^d$ let $X_r(x)$ be the Gaussian random
polynomial of order $r$ defined by
\[
X_r(x)=\sum_{|j|\le r} \xi_j x^j,
\]
where $x^j=\prod_{k=1}^d x_k^{j_k}$ for each multi index
$j=(j_1,\dots,j_d)$, $|j|=\sum_{k=1}^dj_k$, and $(\xi_j)$ are standard
Gaussian random variables.  Then
\[
X_r(\phi)=\sum_{|j|\le r} \xi_j \int_{\R^d} x^j\phi(x)\,dx,\quad \phi\in
\Sc,
\]
defines a corresponding Gaussian random field on $\Sc$. By restricting
to $\Sc_r$ one obtains the order $r$ terms
\[
X_r(\phi)=\sum_{|j|=r} \xi_j \int_{\R^d} x^j\phi(x)\,dx,\quad \phi\in
\Sc_r.
\]
As a field on $\Sc_r$ the polynomial field $X_r$ is $r$-selfsimilar and stationary. Indeed, if
$\phi\in  \Sc_r$ then
\[
X_r(\phi(\cdot+a))=\sum_{|j|=r} \xi_j \int_{\R^d}
(x-a)^j\phi(x)\,dx=\sum_{|j|=r} \xi_j \int_{\R^d}
x^j\phi(x)\,dx.
\]

%%%%%%%%%%%%%%%%%%%%%%%%%%%%%%%%%%%%%%%%

\subsection*{Nondegenerate selfsimilar Gaussian random fields}  

Considering Gaussian $H$-selfsimilar random fields on $\Sc_r$ with
$H<r$, it follows by Theorem \ref{thmdobrushin} that 
\begin{equation}\label{covdobrushin}
C(\phi,\psi)=
\int_{\S^{d-1}}\int_0^\infty \widehat \phi(u\theta)\overline{\widehat
  \psi(u\theta)} u^{-2H-1}\,du\sigma(d\theta),\quad \phi\in \Sc_r,
\end{equation}
and if we specialize to isotropic random fields then the covariance
functional takes the form
\begin{equation}\label{covisotropic}
C(\phi,\psi)=\mathrm{const} \int_{\R^d} \widehat \phi(z)\overline{\widehat
  \psi(z)}\, |z|^{-2H-d}\,dz.
\end{equation}
The most basic case is $H=-d/2$ and $r=0$ combined with a rotationally
symmetric measure $\sigma(d\theta)$. By Parseval's identity, this is
Gaussian white noise $M_d(dx)$ with $M_d(\phi)\sim N(0,\int\phi(x)^2\,dx)$
and $C(\phi,\psi)=\int_{\R^d}\phi(x)\psi(x)\,dx$ (ignoring constants).

If we return to (\ref{covdobrushin}) but restrict the range of
parameters to $-d/2<H<r$, then the covariance may be recast into
\begin{equation}\label{covdobrushinalt}
C(\phi,\psi)= \int_{\R^d\times\R^d} \phi(x)
  \psi(y) |x-y|^{2H}\,\Kc\Big(\frac{x-y}{|x-y|}\Big)\,dxdy,
\end{equation}
where $\Kc$ is an anisotropy weight function on $\S^{d-1}$ defined by 
\[
\Kc(e)=\int_{\S^{d-1}}\int_0^\infty e^{-i r\theta\cdot e}
  u^{-2H-1}\,du\sigma(d\theta),\quad e\in \S^{d-1}.
\]
Recalling from (\ref{defspaceofmeasures}) the setting of indexing
measures in $\Mc_r$ we conclude that, with the exception of
independently scattered white noise, all isotropic selfsimilar
Gaussian random fields are characterized by a covariance functional
$C(\mu,\mu')=\mathrm{Cov}(X(\mu),X(\mu'))$, such that 
\begin{equation}\label{covfunct}
C(\mu,\mu')=\mathrm{const} 
\int_{\R^d\times\R^d}  |x-y|^{2H}\,\mu(dx)\mu'(dy),\quad \mu,\mu'\in\widetilde\Mc.
\end{equation}
For $-d/2<H<0$ the relevant set $\widetilde\Mc\subset\Mc_r,$ consists
of signed measures with finite Riesz-energy. For $0<H<r$ the
moment condition $\int\mu(dy)=0$ enters and we have the
additional representation
\[
C(\mu,\mu')=\mathrm{const} 
\int_{\R^d\times\R^d}  (|x|^{2H}+|y|^{2H}-|x-y|^{2H})\,\mu(dx)\mu'(dy).
\] 
The self-similarity of the model is
equivalent to the second order self-similarity property
$C(\mu_c,\mu_c)=c^{2H} C(\mu,\mu)$. Our final remark in this section
is that because of (\ref{rieszfourier}) and (\ref{covdobrushin}), the
Riesz kernel preserves self-similarity, in the sense 
\begin{equation}\label{covriesz}
C((-\Delta)^{-m/2}\phi,(-\Delta)^{-m/2}\psi)=
\int_{\S^{d-1}}\int_0^\infty \widehat \phi(r\theta)\overline{\widehat
  \psi(r\theta)} u^{-2H-2m-1}\,du\sigma(d\theta).
\end{equation}
Thus, if $X(\phi)$ is selfsimilar with index $H$ then the random field
$Y(\phi)$ defined by $Y(\phi)=X((-\Delta)^{-m/2}\phi)$ for some $m$
with $H+m<r$, is selfsimilar with index H+m, cf.\ \cite{BEK2010} Thm 4.7.

%%%%%%%%%%%%%%%%%%%%%%%%%%%%%%%%%%%%%%%%
%%%%%%%%%%%%%%%%%%%%%%%%%%%%%%%%%%%%%%%%

\section{Gaussian shot noise random fields}

In this section we introduce a wide class of Gaussian selfsimilar
random fields on $\R^d$, generated by white noise and obtained by a
shot noise construction.  Isotropic as well as anisotropic models are
covered.  The white noise is defined on the extended space
$\R^d\times\R_+$ where the additional degree of freedom may be thought
of as a random radius of an euclidean ball located in $\R^d$.  A class of
nonnegative functions in $L_2(\R^d)$ adds further generality to the
model, acting as pulse functions for a shot noise mechanism driven by
the random balls.  The Riesz kernel transform furthermore provides
means of moving from one range of self-similarity indices to
another. In the end all combined we obtain efficient methods to
extract a variety of processes, bridges and membranes from these
Gaussian random fields.

%%%%%%%%%%%%%%%%%%%%%%%%%%%%%%%%%%%%%%%%

\subsection*{Random ball white noise}

For fixed spatial dimension $d\ge 1$ we consider a parameter $\beta$,
such that 
\[
\beta\in (d-1,d)\cup(d,2d),
\]
put 
\begin{equation}\label{defcontrol}
\widetilde\nu_\beta(du)=u^{-\beta-1} du,\; u>0, \quad
\nu(dz)=dx\,\widetilde\nu_\beta(du),\; z=(x,u)\in \R^d\times \R_+, 
\end{equation}
and let $M_\beta(dz)$ be white noise on $\R^d\times \R_+$ defined by
the control measure $\nu(dz)$. Also, with some abuse of notation, we
write $M_d(dz)$ for Gaussian noise with control measure
$\nu(dz)=dx\,\delta_1(du)$, which in this manner is identified with
Gaussian white noise $M_d(dx)$ as introduced earlier.  It is
convenient to let each Gaussian point $(x,u)$ represent a Euclidean
ball $B(x,u)$ in $\R^d$ centered in $x$ with radius $u>0$.  The
general method of evaluating random fields that we adopt in this work
amounts to measure the aggregation of Gaussian mass from all of
$M_\beta(dz)$ as the stochastic integral
\begin{equation}\label{defrandomballs}
X(\mu)=\int_{\R^d\times \R_+} \mu(B(x,u))\,M_\beta(dz),
\end{equation}
where $\mu$ belongs to a suitable class of signed measures.  This
approach is introduced in \cite{KLNS2007} and developed further in
\cite{BEK2010} and \cite{BD2010}. 

As a preparation to help see the origin of self-similarity in these models we
begin with the simplest case of fixed size balls corresponding to
$M_d(dz)$, and consider
\[
X(\mu)=\int_{\R^d\times \R_+} \mu(B(x,u))\,M_d(dz)
=\int_{\R^d} \mu(B(x,1))\,M_d(dx).
\]
This model is Gaussian with covariance functional
\[
C(\mu,\mu')=\int_{\R^d} \mu(B(x,1)) \mu'(B(x,1)) \,dx
=\int_{\R^d\times\R^d} |B(y,1)\cap B(y',1)|\, \mu(dy)\mu'(dy').
\]
The volume $V(|y-y'|))=|B(y,1)\cap B(y',1)|$ of two intersecting balls
only depends on the distance between the center points $y$ and $y'$
and is given by
\begin{equation}\label{defvolfunction}
V(u)=2v_{d-1}\int_{u/2}^1 (1-s^2)^{\frac{d-1}{2}}\,ds,\quad u\le 2,
\end{equation}
and $V(u)=0$ for $u>2$, where $v_d=|B(0,1)|$ is the volume of the unit ball in $\R^d$, see \cite{G2003}.   The one-point
evaluations
\[
X(\delta_t)=\int_{\R^d} \ind_{\{|x-t|\le 1\}}\,M_d(dx),\quad t\in \R^d,
\]
exist and generate a point-wise defined zero mean Gaussian random
field with covariance $C(\delta_t,\delta_{t'})=V(|t-t'|)$.  This
random field does not possess the self-similarity property itself but
if we replace the control measure
$dx\,\delta_1(du)$ with $dx\,\tilde\nu_\beta(du)$ for 
$M_\beta(dz)$ in (\ref{defrandomballs}), then  the covariance is 
\begin{align*}
C(\mu,\mu')&=\int_{\R^d\times\R_+} \mu(B(x,u)) \mu'(B(x,u))\,dx u^{-\beta-1}\,du\\
&=\int_{\R^d\times\R^d} \int_0^\infty
u^dV(|y-y'|/u)\,u^{-\beta-1}du\, \mu(dy)\mu'(dy')\\
&=\mathrm{const} \int_{\R^d\times\R^d} \frac{\mu(dy)\mu'(dy')}{|y-y'|^{\beta-d}},
\end{align*}
which is selfsimilar with index $H=(d-\beta)/2$ according to
(\ref{covisotropic}), assuming $\mu$ and $\beta$ are such that the
integral exists.  As a second type of modification we replace
$\mu(B(x,1))$ in the previous expression for $X(\mu)$ with integration
of $\mu$ with respect to a spatially shifted power law function
$h_\gamma(y)=|y|^{-(d-\gamma)}$, $0<\gamma<d/2$, and consider
\[
X(\mu)=\int_{\R^d}  \int_{\R^d} h_\gamma(y-x)\,\mu(dy)\, M_d(dx).
\]
The heuristic picture of randomly sized overlapping balls in $\R^d$
now changes to one of overlapping pulse functions.  By
(\ref{rieszcomp}), the covariance is found to have the selfsimilar
shape
\[
\mathrm{Cov}(X(\mu),X(\mu'))=\mathrm{const}
\int_{\R^d\times\R^d} \frac{\mu(dy)\mu'(dy')}{|y-y'|^{d-2\gamma}}.
\]
An equivalent interpretation of this particular construction is that
we integrate the Riesz kernel with respect to white noise $M_d(dx)$:
\begin{align}\label{rieszkernelwhitenoise}
\langle M_d,(-\Delta)^{-\gamma/2}\mu(\cdot)\rangle
                    =\int_{\R^d} \int_{\R^d} |y-x|^{-(d-\gamma)}\,\mu(dy)\,M_d(dx).
\end{align}
We emphasize the distinction between the use of the Riesz kernel in
(\ref{rieszkernelwhitenoise}) as opposed to the effect of
Riesz integration by shifting from $\mu$ to
$(-\Delta)^{-m/2}\mu$ in the random balls model in \eqref{defrandomballs},  applying the composition rule (\ref{rieszcomp}), and obtaining
\begin{align*}
C((-\Delta)^{-m/2}\mu,(-\Delta)^{-m/2}\mu')
&=\mathrm{const} \int_{\R^d\times\R^d} \frac{(-\Delta)^{-m/2}\mu(dy)(-\Delta)^{-m/2}\mu'(dy')}{|y-y'|^{\beta-d}}\\
&=\mathrm{const} \int_{\R^d\times\R^d} \frac{\mu(dy)\mu'(dy')}{|y-y'|^{\beta-d-2m}}.
\end{align*}
The range of the self-similarity index in these relations will depend on
a more detailed analysis of which combinations of parameters and
admissible measures one can use, and will be part of the subsequent
results.

%%%%%%%%%%%%%%%%%%%%%%%%%%%%%%%%%%%%%%%%

\subsection*{Shot noise}

We are now in position to introduce a Gaussian shot noise random
field $X_h$ driven by $M_\beta(dz)$ and with a given pulse function
$h$ in $L_2(\R^d)$.  We define the shift and scale mapping
\begin{equation}\label{tauL2}
\tau_z h(y)=h((y-x)/u),\quad z=(x,u)\in \R^d\times\R_+,
\end{equation}
and put
\[
X_h(\mu)=\int_{\R^d \times \R_+}\langle \mu,\tau_z h\rangle\,M_\beta(dz),\quad 
 \langle \mu,\tau_z h\rangle=\int_{\R^d} \tau_zh(y)\,\mu(dy). 
\]
Occasionally we use $\tau_x$ as a short hand notation for $\tau_{(x,1)}$.
The construction of the shot noise then relies on stating proper
assumptions on the class of measures $\mu$ involved and on the class
of admissible pulse functions $h$ for which $X_h$ will exist as a
Gaussian stochastic integral. The shot noise mechanism we investigate here is inspired by similar constructions in \cite{Cag11}.

Following \cite{BEK2010}, for $\beta\not=d$ we let   
\begin{align*}
\Mc^\beta=\Big\{ &\mu\in \Mc : \exists \alpha\; s.t.\;  
\alpha<\beta<d \mbox{   or  } d<\beta<\alpha \\
 &\mbox{  and  }\int_{\R^d\times\R^d}
|y-y'|^{d-\alpha}\,|\mu|(dy)|\mu|(dy')<\infty\Big\}.
\end{align*}
For $d<\beta$ this space of measures is closely related to the set of
measures with finite Riesz energy. Then we combine $\Mc^\beta$  with the
previously introduced sets $\Mc_r$, $r=0,1,\dots$, and put
\[
\Mc^\beta_r=\Mc^\beta\cap \Mc_r.\qquad 
\wt\Mc_\beta=
\begin{cases}
\Mc^\beta, &d<\beta<2d, \\
\Mc^\beta_1,& d-1<\beta<d.
\end{cases}
\]
Let $\Hc_\beta$ be the subset of functions in $L_2(\R^d)$, such that,
for the case $d<\beta<2d$,
\[
\Big|\int_{\R^d} h(x) h(x+y)\,dx\Big|\le \frac{\mathrm{const}}{|y|^{\alpha-d}},\,\qquad
\mbox{all $y\in\R^d$ and $\alpha\in (\beta,2d)$,}
\]
and, for the case $d-1<\beta<d$, 
\[
\Big|\int_{\R^d} h(x)(h(x+y)-h(x))\,dx\Big|\le \mathrm{const}|y|^{d-\alpha},\,\qquad
\mbox{all $y\in\R^d$ and $\alpha\in (d-1,\beta)$}.
\]

\begin{theorem}
\label{thmshotnoise}
Fix $\beta \in (d-1,d)\cap (d,2d)$. Let $M_\beta(dz)$ be the Gaussian
random ball white noise on $\R^d\times\R_+$ with control measure
$\nu(dz)=dx\,\widetilde\nu_\beta(du)$ as defined in
(\ref{defcontrol}).  Assume $h\in \Hc_\beta$ and 
let $H$
denote the parameter
\[
H=\frac{d-\beta}{2}\in 
\begin{cases}
(-d/2,0),& d<\beta<2d, \\
(0,1/2), & d-1<\beta<d.
\end{cases}
\]

\medskip
\noindent
i)   
The shot noise random field 
\[
\mu\mapsto X_h(\mu),\quad \mu\in\wt\Mc_\beta,
\]
is well-defined as a zero mean Gaussian $H$-selfsimilar stochastic
integral with covariance functional
\[
\mathrm{Cov}(X_h(\mu),X_h(\mu'))=\int_{\R^d\times\R^d}
\Kc_h\Big(\frac{y-y'}{|y-y'|}\Big)\,\frac{\mu(dy)\mu'(dy')}{|y-y'|^{\beta-d}},    
\]
where the kernel function $\Kc_h$ is defined on the unit sphere
$\S^{d-1}$ and given by 
\[
\Kc_h(e)=
\begin{cases}
\displaystyle{C_h\int_0^\infty u^{d-1-\beta}\int_{\R^d} h(x)h(x+e/u)\,dxdu}, & d<\beta<2d, \\[10pt]
\displaystyle{C_h\int_0^\infty u^{d-1-\beta}\int_{\R^d} h(x)(h(x+e/u)-h(x))\,dxdu}, & d-1<\beta<d,
\end{cases}
\]
$e\in \S^{d-1}$, for some constant $C_h$.  In particular, if $h$ is
rotationally symmetric on $\R^d$ then $\Kc_h(e)=\Kc_h$ is a constant
and the random field $X_h$ is isotropic.

\medskip
\noindent
ii) Consider the restricted range $d<\beta<2d$. For
the case $d\ge 2$, let $m$ be a real number such that
\[
1<2m<d,\quad 0<d-\beta+2m<2,
\]
and put $H'=H+m$. Assume $(-\Delta)^{-m/2}\mu\in\Mc^\beta$. Then the random field
\[
\mu\mapsto X_h((-\Delta)^{-m/2}\mu)=\int_{\R^d\times\R_+}
\langle (-\Delta)^{-m/2}\mu,\tau_zh\rangle \,M_\beta(dz),
\]
is $H'$-selfsimilar.  For the one-dimensional case, $d=1$,  with
$1<\beta<2$, the random field 
\[
\mu\mapsto X_h((-\Delta)^{-1/2}\mu)=\int_{\R\times\R_+}
\int_\R h((y-x)/u)\mu([y,\infty))\,dy \,M_\beta(dx,du),
\]
is $(3-\beta)/2$-selfsimilar for $\mu$ such that
$\int_y^\infty\mu(dz)\,dy\in \Mc^\beta$. 
\end{theorem}

\begin{proof}
i)  The Gaussian stochastic integral $X_h(\mu)$ exists if and only if
the variance 
\[
\mathrm{Cov}(X_h(\mu),X_h(\mu))
=\int_{\R^d\times\R_+} \langle \mu,\tau_{x,u}h\rangle^2\,dx\,u^{-\beta-1}\,du
\]
is finite. We need to verify that this is the case under the stated assumptions
and establish the explicit form of the covariance functional.   The
proof can be seen as an adaptation of Lemma 2.3 in
\cite{BEK2010} to the case of a shot noise weight function $h$. 

We begin with the case $d<\beta<2d$. Then $\wt\Mc_\beta=\Mc^\beta$.
We introduce the function $g$ defined by 
\[
g(u)=\int_{\R^d} \langle \mu,\tau_{x,u}h\rangle^2\,dx,\quad u>0.
\]
Using Fubini's theorem and homogeneity, we obtain
\begin{equation}\label{proofidentity}
g(u)=u^d\int_{\R^d\times\R^d}
\int_{\R^d}h(x)h(x+(y-y')/u)\,dx\,\mu(dy)\mu(dy').
\end{equation}
Since $h\in \Hc_\beta$ and $\mu\in \Mc^\beta$ we can find $\alpha\in
(\beta,2d)$, such that
\begin{equation}\label{proofbound}
0<g(u)\le \mathrm{const}\, u^\alpha \int_{\R^d\times\R^d}
\frac{|\mu|(dy)|\mu|(dy')}{|y-y'|^{\alpha-d}}<\infty.
\end{equation}
On the other hand, using H\"older's inequality and $\mu\in \Mc$, it
follows from (\ref{proofidentity}) that $g(u)\le \|h\|_2\,
\|\mu\|^2\,u^d$, so that %and hence we can find $C_h$ such that 
\begin{equation}\label{proofbound2}
0<g(u)\le \mathrm{const}\, \min(u^\alpha,u^d)
\end{equation}
and hence
\[
\int_0^\infty g(u)\, u^{-\beta-1}du= 
\int_{\R^d\times\R_+} \langle\mu,\tau_z h\rangle^2\,\nu(dz)<\infty.
\]
Next we may replace $g$ in the left-hand side integral by the integral
expression in (\ref{proofidentity}) and apply a change of variables, to obtain 
\[
\int_{\R^d\times\R_+} \langle\mu,\tau_z h\rangle^2\,\nu(dz)
 = \int_{\R^d\times\R^d}\Kc_h\Big(\frac{y-y'}{|y-y'|}\Big)
\frac{\mu(dy)\mu(dy')}{|y-y'|^{\beta-d}}
\]
with the desired function $\Kc_h$, as stated in the theorem.  By
(\ref{covdobrushinalt}), this is the covariance functional for a
selfsimilar Gaussian model with self-similarity index
$H=-(\beta-d)/2<0$.  

For the remaining case $d-1<\beta<d$ in statement
i), we have $\mu\in \Mc_1$ and hence $\int_{\R^d}\mu(dx)=0$. Thus, we
may replace (\ref{proofidentity}) by
\[
g(u)=u^d\int_{\R^d\times\R^d}
\int_{\R^d}h(x)(h(x+(y-y')/u)-h(x))\,dx\,\mu(dy)\mu(dy').
\]
Then we use the relevant property of $h\in \Hc_\beta$ for this range
of the parameter $\beta$ to obtain an $\alpha\in (d-1,\beta)$, such
that the bounds in (\ref{proofbound}) and (\ref{proofbound2}) are
preserved.  In parallel with the previous case it remains to integrate
over $u$ to obtain the covariance functional, which now yields a
self-similarity index $H$ in the range $0<H<1/2$.

To prove part ii) of the theorem we begin with the case $d\ge 2$, take
$\beta$ and $m$ as specified, and consider the function
\[
g_m(u)=\int_{\R^d} \langle (-\Delta)^{-m/2}\mu,\tau_{x,u}h\rangle^2\,dx,\quad u>0,
\]
for $h\in \Hc_\beta$ and $(-\Delta)^{-m/2}\mu\in \Mc^\beta$.
Using the notation
\[
V_h(y)=\int_{\R^d} h(x)h(x+y)\,dx, \quad y\in \R^d,
\]
we have
\[
g_m(u)=u^d \int_{\R^d\times\R^d} V_h((y-y')/u)\,
(-\Delta)^{-m/2}\mu(dy)(-\Delta)^{-m/2}\mu(dy').  
\]
By using $h\in \Hc_\beta$ and H\"older's inequality as in the proof of
part 1), we find that $g_m$ satisfies relation (\ref{proofbound2}) for
some $\alpha$ with $\beta<\alpha<2d$, which implies that the
covariance functional
\[
C(\mu,\mu)=\int_0^\infty g_m(u)\, u^{-\beta-1}du= 
\int_{\R^d\times\R_+} \langle (-\Delta)^{-m/2}\mu,\tau_z h\rangle^2\,\nu(dz)
\]
is finite.  Moreover, by a change of variable and relation (\ref{rieszcomp}),
\[
g_m(u)=u^d \int_{\R^d} V_h(w/u)\, \int_{\R^d\times\R^d} 
\frac{\mu(dy) \mu(dy')}{|y-y'+w|^{d-2m}}\,dw.
\]
Thus,
\[
C(\mu,\mu)=
\int_{\R^d} \frac{1}{|w|^{\beta-d}} \Kc_h\Big(\frac{w}{|w|}\Big)
\frac{\mu(dy) \mu(dy')}{|y-y'+w|^{d-2m}}\,dw,
\]
where
\[
\Kc_h(e)=\mathrm{const}\int_0^\infty u^{d-\beta-1}V_h(e/u)\,du, \quad e\in\S^{d-1},
\]
is a finite function on the unit sphere. Clearly,
\[
C(\mu_c,\mu_c)=c^{2H'}\,C(\mu,\mu),\quad H'=\frac{d-\beta+2m}{2}. 
\]
For $d=1$, the arguments are parallel and lead to the representation
\[
C(\mu,\mu)=\int_{\R\times\R} 
\int_0^\infty u^{-\beta}V_h\Big(\frac{y-y'}{|y-y'|}\frac{1}{u}\Big)\,du\,
\frac{\mu([y,\infty)) \mu([y',\infty))}{|y-y'|^{\beta-1}}\,dydy',
\]
which scales with self-similarity index of order $(3-\beta)/2\in (1/2,1)$.
\end{proof}

\begin{theorem}
\label{thmsingularshot}
Let $M_d(dx)$ be Gaussian white noise on $\R^d$ with control
measure $dx$. For $\beta\in (d-1,d)\cup (d,2d)$, put
$H=(d-\beta)/2$ and let $h_\beta$ and $h_\beta^+$ be functions defined
by
\[
h_\beta(x)=|x|^{-\beta/2},\quad x\in \R^d,\qquad 
h_\beta^+(x)=x_+^{-\beta/2},\quad x\in \R, \quad x_+=x\vee 0. 
\]

\medskip\noindent
i) The Gaussian random field
\[
\mu\mapsto X(\mu)=\int_{\R^d} \langle \mu, \tau_x
h_\beta\rangle \,M_d(dx), \quad \mu\in \wt\Mc_\beta, 
\]
is  $H$-selfsimilar with covariance 
\[
\mathrm{Cov}(X(\mu),X(\mu'))=\mathrm{const}  \int_{\R^d\times\R^d}
 |y-y'|^{2H} \mu(dy)\mu'(dy').
\]
For $d-1<\beta<d$ this may be written 
\[
\mathrm{Cov}(X(\mu),X(\mu'))=C_+ \int_{\R^d\times\R^d}
 \Big(|y|^{2H}+|y'|^{2H} -|y-y'|^{2H}\Big)\mu(dy)\mu'(dy')
\]
with a positive constant $C_+$.

\medskip\noindent ii) Restricting to $d\ge 2$ and $d<\beta<2d$, let
$m$ be a real number such that
\[
0<2m<d,\quad 0<d-\beta+2m<2.
\]
For $\mu$ such that $(-\Delta)^{-m/2}\mu\in \Mc^\beta$ we have  
\begin{align*}
\mu\mapsto 
&\int_{\R^d} \langle
(-\Delta)^{-m/2}\mu,\tau_xh_\beta\rangle\, M_d(dx)\\  
&=\mathrm{const}\int_{\R^d}\int_{\R^d}|y-x|^{H'-d/2} \mu(dy)\, M_d(dx),
\end{align*} 
and this map defines a selfsimilar Gaussian random field with
self-similarity index $H'=(d-\beta)/2+m\in (0,1)$.   For the case $d=1$ and
$1<\beta<2$, the random field 
\begin{align*}
\mu\mapsto 
\int_{\R} \langle
(-\Delta)^{-1/2}\mu,\tau_x h^+_\beta\rangle\, M_1(dx)
=\mathrm{const}\int_{\R}\int_{\R}(y-x)_+^{H'-1/2} \mu(dy)\, M_1(dx),
\end{align*}
is $H'$-selfsimilar with $H'=(3-\beta)/2\in (1/2,1)$.  Also, 
\begin{equation}\label{thmBM}
\mu\mapsto 
\int_{\R} (-\Delta)^{-1/2}\mu(x)\, M_1(dx)
=\int_{\R}  \int_x^\infty \mu(dy)\, M_1(dx)
\end{equation}
is $1/2$-selfsimilar. 

\end{theorem}
\begin{proof}
To prove $i)$ we need to establish that
\[ 
\mathrm{Cov}(X(\mu),X(\mu))=\int_{\R^d} \langle \mu, h_\beta\rangle^2\,dx
\]
has the required properties.  Indeed, there is a constant $c_\beta$
such that
\[
\int_{\R^d} \langle \mu, h_\beta\rangle^2\,dx
=c_\beta \int_{\R^d\times\R^d}
|y-y'|^{d-\beta}\,\mu(dy)\mu(dy').
\]
Here,
\[
c_\beta=\int_{\R^d} h_\beta(x)h_\beta(x+e)\,dx <\infty,
\]
some $e\in \S^{d-1}$, for the case $d<\beta<2d$, and, using $\int\mu(dy)=0$, 
\[
c_\beta= 
\int_{\R^d} h_\beta(x)(h_\beta(x+e)-h_\beta(x))\,dx<\infty
\]
for the case $d-1<\beta<d$. To prove ii) for $d\ge 2$, 
$d<\beta<2d$, and $m$ as specified, we have by (\ref{rieszcomp}),
\begin{align*}
&\int_{\R^d} \langle
(-\Delta)^{-m/2}\mu,\tau_xh_\beta\rangle\, M_d(dx)\\  
&\qquad \qquad =\mathrm{const}\int_{\R^d}\int_{\R^d\times\R^d}
\frac{\mu(dy)}{|y-w|^{d-m}}\,\frac{dw}{|w-x|^{\beta/2}}\, M_d(dx)\\
&\qquad \qquad =\mathrm{const}\int_{\R^d}\int_{\R^d}
\frac{\mu(dy)}{|y-x|^{\beta/2-m}}\, M_d(dx)\\
&\qquad \qquad =\mathrm{const}\int_{\R^d}\int_{\R^d}|y-x|^{H'-d/2} \mu(dy)\, M_d(dx),
\end{align*} 
and we can check as before that the covariance is finite under the
given assumptions. The proof for the one-dimensional case, which uses
$m=1$, is analogous. 
\end{proof}

%%%%%%%%%%%%%%%%%%%%%%%%%%%%%%%%%%%%%%%%

\subsection*{Random cylinder Gaussian fields}

The purpose of this subsection is to show that Brownian sheets models
are naturally included in the general framework of selfsimilar random
fields, and that they emerge from expanding the white noise
construction in Theorem \ref{thmshotnoise} based on random balls to
one based on random cylinders.  In the interest of not burdening our main
result Theorem \ref{thmshotnoise} with additional notation and
variations we have chosen to present these results in a separate
subsection and in a less formal manner.

We define random cylinder white noise on the product space $\R^d\times
\R_+^p$, $1\le p\le d$, equipped with a control measure that allows us
to think of the noise as Gaussian fluctuations of overlapping random
cylinders. The special case $p=1$ is the random balls white noise.   
For a given spatial integer dimension $d$ we consider an arbitrary
partition $d^\pi=(d_1,\dots,d_p)$ of $d$, $d=\sum_{i=1}^p d_i$.  Any
point $x\in \R^d=\prod_{i=1}^p \R^{d_i}$ has the representation
$x^\pi=(x^1,\dots,x^p)$, where $x^i\in\R^{d_i}$ for $1\le i\le p$.
Given a set of parameters $\widetilde\beta=(\beta_1,\dots,\beta_p)$
such that either $d_i < \beta_i < 2d_i$, $1\le i\le p$, or $d_i-1 <
\beta_i < d_i$, $1\le i\le p$, we 
define a measure $\widetilde\nu_\beta(du)$ on $\R^p_+=[0,\infty)^p$ by
\begin{equation}\label{cylinderrandomradius}
\widetilde\nu_\beta(du)=\prod_{i=1}^p u_i^{-\beta_i-1}\,du_i.
\end{equation}
The scaling relation 
\[
\widetilde\nu_\beta(c\, du)=c^{-\beta}\,\widetilde\nu_\beta(du),\quad
c>0, \qquad \beta=\sum_{i=1}^p\beta_i,
\]
holds.  Let $M_\beta(dz)$ be a Gaussian measure on $\R^d\times
\R_+^p$ defined by the intensity measure $\nu(dz)=dx\,\widetilde
\nu_\beta(du)$.  With each point $z=(x,u)$ we associate a shift and
scale operator $\tau_z h:\R^d\mapsto \R^{d+p}$ acting on functions
$h\in L_2(\R^d)$, by
\[
\tau_zh(y)=h((y^1-x^1)/u_1,\dots,(y^p-x^p)/u_p)\,. 
\]
In particular, letting $h$ be the indicator function of the partition
unit ball $C(0,1)=\{y^\pi\in\R^d: |y^i|_{d_i}\le 1, 1\le i\le p\}$,
where $|\cdot|_k$ is the euclidean norm in $\R^k$, it follows that
$\tau_zh$ with $z=(x,u)$ is the indicator function of the random
cylinder $C(x,u)$ with center point $x\in \R^d$ and partition radius
$u$, that is
\[
C(x,u)=\{y^\pi\in \R^d: |y^i-x^i|_{d_i}\le u_i, 1\le i\le p\}, 
\quad x^\pi\in \R^d, u\in R_+^p.
\] 
The map $\tau_z$ has the invariance property
\[
\tau_z h(cy)=\tau_{z/c} h(y),\quad y\in \R^d,\,
z\in\R^{d+p},\quad c>0.
\] 
In analogy to the shot noise model we define the cylinder random
field by
\[
X_h(\mu)=\int_{\R^d\times\R^p_+} \langle\mu,\tau_zh\rangle\,M_\beta(dz).
\]
By proper modifications of the arguments given in the previous
sections one can show that the generalized random field $\mu\mapsto
X(\mu)$ is well-defined for a suitably restricted class of measures.
For simplicity we focus on the simplest case $h(y)=\ind_{\{|y|\le
  1\}}$ in the rest of this section, and hence consider 
\[
X(\mu)=\int_{\R^d\times\R^p_+} \mu(C(x,u))\,M_\beta(dz).
\]
The covariance functional is
\begin{align*}
C(\mu,\mu')&=\int_{\R^d\times\R^d}\mu(dy)\mu'(dy') \int_{\R^p_+}|C(y,u)\cap
C(y',u)|\,\prod_{i=1}^p u_i^{-\beta_i-1}\,du_i\\
&=\mathrm{const}\int_{\R^d\times\R^d}\mu(dy)\mu'(dy') \prod_{i=1}^p
|y^i-{y'}^i|_{d_i}^{d_i-\beta_i}.
\end{align*}
Put 
\[
H=\sum_{i=1}^p(d_i-\beta_i)=d-\beta\in (-d/2,0)\cap (0,1/2).
\]
Then $C(\mu_c,\mu'_c)=c^{2H}C(\mu,\mu')$ and it follows that the
cylinder random field is selfsimilar with index $H$. To recognize this
model as an instance of Theorem \ref{thmdobrushin}, let $\Kc$ be the
function on $\S^{d-1}$ defined such that if $e\in\S^{d-1}$ has 
decomposition $e^\pi=(e^1,\dots,e^p)$, then  
\[
\Kc(e)=\prod_{i=1}^p |e^i|_{d_i}^{d_i-\beta_i},\quad
e=e^\pi\in \S^{d-1}.
\]
Then 
\[
C(\mu,\mu')
=\mathrm{const}\int_{\R^d\times\R^d}\mu(dy)\mu'(dy') 
\Kc\Big(\frac{y-y'}{|y-y'|_d}\Big)\, |y-y'|_d^{d-\beta}.
\]    

%%%%%%%%%%%%%%%%%%%%%%%%%%%%%%%%%%%%%%%%
%%%%%%%%%%%%%%%%%%%%%%%%%%%%%%%%%%%%%%%%

\section{Extracting Gaussian processes from the random fields}

The main tool for extracting random processes indexed by points on
the real line or points in Euclidean space, from abstract random
fields $X(\mu)$ indexed by measures $\mu$, will be to evaluate the
random fields using specifically chosen families of measures, such as 
$\mu_t=\delta_t-\delta_0$, $0,t\in \R^d$, $d\ge 1$.

%%%%%%%%%%%%%%%%%%%%%%%%%%%%%%%%%%%%%%%%

\subsection*{Fractional Brownian motion}

Fractional Brownian motion on $\R^d$ is a para\-metrized class of
pointwise defined, centered Gaussian random fields $B_H(t)$, $t\in R^d$,
defined by the covariance functional 
\[
{\rm
  Cov}(B_H(s),B_H(t))=\frac{1}{2}(|s|^{2H}+|t|^{2H}-|t-s|^{2H}), \quad
s,t\in \R^d,
\]
where the parameter $H$, called the Hurst index, ranges over $0<H<1$
and is the self-similarity index in the sense of
$\{B_H(ct)\}\stackrel{d}{=} \{c^H B_H(t)\}$, $c>0$. The case $H=1/2$ is known
as L\'evy Brownian motion. See \cite{CohenIstas} and \cite{SamTaqqu} for
the general theory of such processes.

Next we show how to obtain $B_H$ from the selfsimilar Gaussian random
fields constructed in Theorem \ref{thmshotnoise}.  In part i) of the
theorem we take $\beta$ such that $d-1<\beta<d$ and a rotationally
symmetric function $h$ on $\R^d$ such that $h\in \Hc_\beta$. Then
$\mu=\delta_t-\delta_0\in \wt\Mc_\beta$, and the map 
\[
t\mapsto X_h(\delta_t-\delta_0)=\int_{\R^d\times\R_+}
(h((t-x)/u)-h(-x/u))\,M_\beta(dx,du),  
\]
defines a zero mean Gaussian random field with
covariance function
\[
C(s,t)=\rm{Cov}(X_h(\delta_s-\delta_0),X_h(\delta_t-\delta_0))
\]
given by 
\begin{align*}
C(s,t)&=\mathrm{const}
\int_{\R^d\times\R^d}|y-y'|^{d-\beta}(\delta_s-\delta_0)(dy)(\delta_t-\delta_0)(dy')\\
&=\mathrm{c_h} \,(|t|^{2H}+|s|^{2H}-|t-s|^{2H}), 
\end{align*}
which is a multiple of fractional Brownian motion with Hurst index $H\in
(0,1/2)$.  In particular, with $h(y)=\ind_{\{|y|\le 1\}}$ we have
\[
X_h(\delta_t-\delta_0)=\int_{\R^d\times \R_+}
(\delta_t(B(x,u))-\delta_0(B(x,u)))\,M_\beta(dx,du).
\]
This representation of $B_H(t)$ for the case $0<H<1/2$ is
discussed in \cite{BEK2010} and may be recognized as a so called $(2,H)$-Takenaka field $B_H(t)=M_\beta(V_t)$, where
\begin{align*}
V_t&=\{\hbox{all spheres separating $0$ and $t$}\}\\
&=\{(x,r):|x|\le r\} \setDiff \{(x,r):|x-t|\le r\},
\end{align*}
where $\setDiff$ denotes the symmetric difference of two sets in $\R^d$, see \cite{SamTaqqu}.  Next, in part ii) of Theorem \ref{thmshotnoise}
we consider $d\ge 2$ and take $\beta$ and $m$ such that $d<\beta<2d$,
$0<d-\beta+2m<2$ and $0<2m<d$, and pick $h\in \Hc_\beta$ again
rotationally symmetric. To show that the measure
\[
(-\Delta)^{-m/2}(\delta_t-\delta_0)(dy)=C_{m,d}
\Big(\frac{1}{|t-y|^{d-m}}-\frac{1}{|y|^{d-m}}\Big)\,dy
\]
belongs to $\Mc_\beta$, we observe
\begin{align*}
(-\Delta)^{-m/2}&(\delta_t-\delta_0) \ast
(-\Delta)^{-m/2}(\delta_t-\delta_0)(dy)\\
&=C_{2m,d}\Big(\frac{2}{|y|^{d-2m}}-\frac{1}{|t+y|^{d-2m}}
-\frac{1}{|t-y|^{d-2m}}\Big)\,dy
\end{align*}
and 
\[
\int_{\R^d} \frac{1}{|y|^{\beta-d}}\Big|\frac{2}{|y|^{d-2m}}-\frac{1}{|t+y|^{d-2m}}
-\frac{1}{|t-y|^{d-2m}}\Big|\,dy<\infty.
\]
Thus, under the stated assumptions,
\begin{equation}\label{fbmshotnoise}
t\mapsto \int_{\R^d\times\R_+}
\langle(-\Delta)^{-m/2}(\delta_t-\delta_0)),\tau_zh\rangle
\,M_\beta(dz), \quad t\in \R^d,
\end{equation}
is a multiple of fractional Brownian motion with Hurst index
$H'=H+m\in (0,1)$.  As an explicit example, with $h(y)=\ind_{\{|y|\le 1\}}$ we
can find a constant $C$, such that
\[
B_{H'}(t)\stackrel{d}{=} C
\int_{\R^d\times \R_+} 
\int_{B(x,u)}\Big(\frac{1}{|t-y|^{d-m}}-\frac{1}{|y|^{d-m}}\Big)\,dy\,M(dx,du),
\]
where $M(dx,du)$ is Gaussian white noise on $\R^d\times\R_+$ with
control measure $\nu(dx,du)=dx\,u^{2H'-d-1-2m}du$.  The special choice
of parameters $d-\beta+2m=1$ with $1<2m<d$, for which $H'=1/2$, shows
that L\'evy Brownian motion is covered by this construction. In
particular, letting $M(dx,du)$ have control measure
$\nu(dx,du)=dx\, u^{-d-2m}du$, 
\[
B_{1/2}(t)\stackrel{d}{=} C
\int_{\R^d\times \R_+} 
\int_{B(x,u)}\Big(\frac{1}{|t-y|^{d-m}}-\frac{1}{|y|^{d-m}}\Big)\,dy\,M(dx,du).
\]
Our corresponding result for dimension $d=1$ is less general in the
sense that $1<\beta<2$, $m=1$ and $H'=(3-\beta)/2\in (1/2,1)$. Random balls
representations for the one-dimensional model with this range of Hurst
index have been studied earlier, see e.g. \cite{KLNS2007},
\cite{KT2008}.  Now
\begin{equation}\label{rieszonedim}
(-\Delta)^{-1/2}(\delta_t-\delta_0)=\int_x^\infty(\delta_t-\delta_0)(dy)=1_{[0,t]}(x).
\end{equation}
Hence, letting $M(dx,du)$ be a Gaussian measure on $\R\times\R_+$ with
control measure $\nu(dx,du)=dx\,u^{2H'-4}\,du$,  
\[
B_{H'}(t)\stackrel{d}{=} C 
\int_{\R\times\R_+} \int_0^t h((y-x)/u)\,dy\,M(dx,du).
\]

We conclude this subsection by comparing the representations of
fractional Brownian motion obtained above with those we get by taking
$\mu=\delta_t-\delta_0$ in Theorem \ref{thmsingularshot}. For
$d-1<\beta<d$ this choice of $\mu$ in  Theorem \ref{thmsingularshot} i), 
generates the map
\begin{align*}
t\mapsto \int_{\R^d} \langle \delta_t-\delta_0, h_\beta\rangle\, M_d(dx) 
&=\int_{\R^d} (h_\beta(t-x)-h_\beta(-x))M_d(dx)\\
&=\int_{\R^d} (|t-x|^{H-d/2}-|x|^{H-d/2}))M_d(dx)
\end{align*}
for $H\in (0,1/2)$, which we recognize as the so called well-balanced
representation of fractional Brownian motion.  By replacing $h_\beta$
with $h_\beta^+$ for the case $d=1$, we obtain the classical
Mandelbrot and van Ness representation
\begin{equation}\label{mandelbrotness}
B_H(t)\stackrel{d}{=}\int_{\R^d}
((t-x)_+^{H-1/2}-(-x)_+^{H-1/2})\,M(dx),\quad t\ge 0.
\end{equation}
for $0<H<1/2$.  Similarly, Theorem \ref{thmsingularshot} ii) with
$\mu=\delta_t-\delta_0$ also yields a pointwise well-defined random
process on $\R^d$ given by
\begin{align*}
t\mapsto
&\int_{\R^d}\Big(\frac{1}{|t-y|^{d-m}}-\frac{1}{|y|^{d-m}}\Big)
\frac{1}{|y-x|^{\beta/2}}\,dy\,M_d(dx)\\
&=\mathrm{const}\int_{\R^d}\Big(|t-y|^{H'-d/2}-|y|^{H'-d/2}\Big)\,M_d(dx),
\end{align*}
which again is the well-balanced representation of fractional Brownian
motion with Hurst index $H'\in (0,1)$. Finally, the case $d=1$ in
Theorem \ref{thmsingularshot} applies the one-sided pulse function
$h(x)=x_+^{-\beta/2}$ on the real line, and hence extends the
Mandelbrot and van Ness representation (\ref{mandelbrotness}) to the
entire range of Hurst index $0<H<1$.  In particular, by (\ref{thmBM})
and (\ref{rieszonedim}),  
\begin{equation}\label{BM}
W_t=\int_\R (-\Delta)^{-1/2}(\delta_t-\delta_0)(x)\,M_1(dx),\quad t\ge 0,
\end{equation}
is Brownian motion.

%{Harmonizable representation}
%\[
%B_H(t)=C_H\int_{\R^d}\frac{e^{-it\cdot x}-1}{|x|^{d/2+H}}(M_1(dx)+iM_2(dx))
%\]

%%%%%%%%%%%%%%%%%%%%%%%%%%%%%%%%%%%%%%%%

\subsection*{Examples of the cylinder model} 

a) $\nu_1=d$: This is the heavy-tailed, one-parameter random balls
model with $r=1$ and $d<\beta<2d$, for which
\[
C(\mu,\eta)=V_d\int_{\R^d \times \R^d}\mu(dy)\eta(dy') 
|y-y'|_d^{d-\beta}.
\]
b) The case $r=d$, $\nu_1=\dots =\nu_d=1$ gives a non-symmetric random
sheets model with $d$ parameters, $1 < \beta_i < 2$, such that
\[
C(\mu,\eta)=\int_{\R^d \times \R^d}\mu(dy)\eta(dy') 
\prod_{i=1}^d C_i|y_i-y'_i|_1^{1-\beta_i}.
\]
Take product measures $\mu(dy)=\prod_{i=1}^d \mu_i(dy_i)$ and 
$\eta(dy)=\prod_{i=1}^d \eta_i(dy_i)$ to obtain
\[
C(\mu,\eta)=\prod_{i=1}^d C_i\int_{\R \times \R}\mu_i(dy_i)\eta_i(dy'_i) 
|y_i-y'_i|_1^{1-\beta_i}.
\]
In particular, $\mu_i(A)=\int_0^{t_i} \Ind_A(y)\,dy$ and
$\eta_i(A)=\int_0^{s_i} \Ind_A(y)\,dy$, yields
\[
C(\mu,\eta)=\prod_{i=1}^d C_i\int_0^{t_i}\int_0^{s_i}\frac{dy_idy'_i}
{|y_i-y'_i|_1^{\beta_i-1}} =\prod_{i=1}^d C'_i
\Big(|t_i|_1^{3-\beta_i}+|s_i|_1^{3-\beta_i}-|t_i-s_i|_1^{3-\beta_i}\Big).
\]

%%%%%%%%%%%%%%%%%%%%%%%%%%%%%%%%%%%%%%%%

\subsection*{The Gaussian free field} 

The choice of parameters $H=-d/2+1$ and $r=1$ for the isotropic case
in Theorem \ref{thmdobrushin} is sometimes referred to as the Gaussian
free field.  By (\ref{covfunct}), the case $d=1$, $H=1/2$ has
covariance functional 
\[
C(\mu,\mu')=\mathrm{const} 
\int_{\R\times\R}  |x-y|\,\mu(dx)\mu'(dy),\quad \mu,\mu'\in\widetilde\Mc.
\]
With $\mu=\delta_t-\delta_0$ and $\mu'=\delta_s-\delta_0$ this gives
Brownian motion:
\[
\mathrm{const} 
\int_{\R\times\R}
|x-y|\,(\delta_t-\delta_0)(dx)(\delta_s-\delta_0)(dy)
=\mathrm{const}\, \min(s,t).
\]
For $d\ge 3$ we have $-d/2<H<0$, which is a case covered by Theorem
\ref{thmshotnoise} i) with $\beta=2(d-1)$ and $h\in \Hc_\beta$
rotationally symmetric.  The control measure of the driving Gaussian
random balls white noise is $\nu(dx,du)=dx\, u^{-2d+1}du$.

The remaining case $d=2$, $H=0$ is not included in Theorem \ref{thmdobrushin}.
However, white noise $M_2(dx)$ for $d=2$ has $H=-1$ and hence it is
natural to consider for the free field
\[
X(\phi)=M_2((-\Delta)^{-1/2}\phi),\quad \phi\in \Sc_1,
\]
with covariance functional
\begin{align*}
C(\phi,\psi)&=\int_{\R^2} (-\Delta)^{-1/2}\phi(x)
(-\Delta)^{-1/2}\psi(x)\,dx\\
&=\int_{\R^2} \phi(x)
(-\Delta)^{-1}\psi(x)\,dx,
\end{align*}
where 
\[
(-\Delta)^{-1}\psi(x)=\int_{\R^2} \psi(y)G(x-y)\,dy
\]
and $G(x)$ is the Green's function of Brownian motion in $d=2$. Thus,
\[
C(\phi,\psi)=-\int_{\R^2 \times \R^2} \phi(x)\psi(y)\,\log (|x-y|)\,dxdy,
\]
which is known as de Wijs random field, see \cite{besagmondal2005} for
a background.

%%%%%%%%%%%%%%%%%%%%%%%%%%%%%%%%%%%%%%%%
\subsection*{Generalized Gaussian bridges}

For a continuous Gaussian process $X=(X_t)_{t \in [0,T]}$ and a
signed finite Borel measure $a$ on an interval $[0,T]$ of the real
line, denote by $X^{(a)}=(X^{(a)}_t)_{t \in [0,T]}$ the process $X$
conditioned on the event that $a(X) = \int_0^T X_t \,a(dt) = 0$.  Such
generalized Gaussian bridges are studied in \cite{Ali02},
\cite{Gor14}, and \cite{Sot14}.  It is shown that the conditioned
process $X^{(a)}$ admits a representation of the form
\[ 
X^{(a)}_t = X_t - f_{(a)}(t) a(X), \qquad 0 \leq t \leq T, 
\]
for a suitable continuous function $f_{(a)}: [0,T] \rightarrow \R$.
For example, Brownian motion $W=(W_t)_{t \in [0,T]}$ conditioned on
$W_1 = 0$ is obtained from $a = \delta_1$ and yields the Brownian
bridge $B=(B_t)_{t \in [0,T]}$ with representation $B_t = W_t - t
W_1$, $0 \leq t \leq 1$.  

To obtain the Brownian bridge $B$ in our setting of extracting
Gaussian processes from random fields, we apply relation (\ref{thmBM})
of Theorem \ref{thmsingularshot} with the special choice 
\equ[E:M1]{
  \mu = \delta_t - \delta_0 - t(\delta_1 - \delta_0) \in \Mc_2. 
}  
and use the linearity of the mapping $\mu \mapsto M((-\Delta)^{-1/2}\mu)$
to see
\begin{equation}\label{brownianbridge}
B_t= M_1((-\Delta)^{-1/2}(\delta_t - (1-t)\delta_0 - t\delta_1)),
\quad 0\le t\le 1.
\end{equation}
This observation generalizes as follows: consider $X_t =
M((-\Delta)^{-1/2} \mu_t)$, for an unspecified family of measures
$\mu_t\in \Mc_1$, $0 \leq t \leq T$, and assume that $(X_t)_{t \in
  [0,T]}$ is continuous on $[0,T]$ almost surely. Let $a$ be a signed
finite Borel measure on $[0,T]$ and for suitable continuous $f^{(a)}:
[0,T] \rightarrow \R$, define
\[
\mu_t^{(a)}(A) = \mu_t(A) - f^{(a)}(t) \int_0^T \mu_s(A) \,a(ds),
\] 
for Borel sets $A \subset [0,1]$.  Then the conditioned process
$X^{(a)}$ has the representation
\[
  X_t^{(a)} = X_t - f^{(a)}(t) a(X)= M_1((-\Delta)^{-1/2} \mu_t^{(a)}).
\]
Indeed,
\begin{align*}
 X_t^{(a)} &= \int_\R \int_x^\infty \mu_t(dy) M_1(dx) - f^{(a)}(t) \int_0^T \int_\R \int_x^\infty \mu_s(dy)\, M_1(dx)\, a(ds) \\
    &= \int_\R \int_x^\infty \left[ \mu_t - f^{(a)}(t) \int_0^T \mu_s
   a(ds) \right](dy)\, M_1(dx). 
\end{align*} 
As a concrete example we consider the Brownian bridge $B$ with
representation (\ref{brownianbridge}) conditioned
to have vanishing Lebesgue measure on $[0,1]$, i.e., $a(dt) = dt$. In
\cite{Gor14}, the resulting conditioned process $B^{(a)}$ is called the zero
area Brownian bridge, and is shown to satisfy
\[ 
B^{(a)}_t = B_t - 6t(1-t) a(B). 
\] 
Here, we recover this relation as
$B^{(a)}(t) = M_1((-\Delta)^{-1/2}\mu^{(a)}_t)$
with
\begin{align*}
  \mu_t^{(a)}(A) &= \mu_t(A) - 6t(1-t) \int_0^T \mu_s(A) ds \\
    &= \delta_t(A) - \delta_0(A) (1-4t+3t^2) + \delta_1(A) (2t - 3t^2) + |A| 6(t^2-t)
\end{align*}
for Borel sets $A \subset [0,1]$.  One can check that $\mu_t^{(a)}\in \Mc_3$.

%%%%%%%%%%%%%%%%%%%%%%%%%%%%%%%%%%%%%%%%

\subsection*{Volterra processes}

The representation of the Brownian bridge in (\ref{brownianbridge})
%$B_t=M_1((-\Delta)^{-1/2}\mu_t)$ with $\mu_t$ defined in \eqref{E:M1},
involves the measure densities 
\[
(-\Delta)^{-1/2}(\delta_t-(1-t)\delta_0-t\delta_1)(x)=
(1-t)\ind_{[0,t]}(x)-t\ind_{(t,1]}(x),\quad t \in [0,1],
\]
supported on $[0,1]$.  However, the Brownian bridge also admits a
representation as a Volterra process of the form
\[ 
B_t = \int_0^t \frac{1-t}{1-x}\, M_1(dx). 
\]
The question arises if we are able to define measures $(\mu_t)_{t \in
  [0,1]}$ on $\R$ such that the support of $(-\Delta)^{-1/2}\mu_t$ is a subset of 
$[0,t]$ and $B_t = M_1((-\Delta)^{-1/2}\mu_t)$.

Let $X=(X_t)_{t \in [0, \infty)}$ be a Gaussian Volterra process with
\[ 
X_t = \int_0^t K(t,x) \,M_1(dx) 
\]
and assume that the kernel $K$ is defined on $\R \times \R$ with
$K(t,x) = K(t,x)\, \Ind(0 < x \leq t)$. Moreover, assume that $K(t,
\cdot)$ is continuous from the left and has limits from the right and has
finite total variation. Then the measures $(\mu_t)_{t \in \R}$ defined
by $\mu_t((-\infty, x)) = - K(t,x)$ are admissible in the sense
$\mu_t\in \Mc_1$,  and $(-\Delta)^{-1/2}\mu_t$ is supported on $[0,t]$. If
$K(t, x)$ is differentiable for $0 < x < t$ then
\[ 
\mu_t(dx) = K(t,t) \,\delta_t(dx) - K(t,0+) \,\delta_0(dx) -
\frac{\partial}{\partial x}K(t,x) \, \Ind_{(0,t]}(x)\, dx. 
\]
By defining $\mu_t$ in this manner it follows that 
$X_t = M_1((-\Delta)^{-1/2}\mu_t)$. In fact, 
\begin{align*}
  M_1((-\Delta)^{-1/2}\mu_t) 
    &= \int_0^\infty \int_x^\infty \mu_t(dy) \,M_1(dx) 
= \int_0^\infty \mu_t([x, \infty)) \,M_1(dx)
\end{align*}
and so, since $\mu_t(\R) = 0$,
\begin{align*}
  M_1((-\Delta)^{-1/2}\mu_t) &= \int_0^\infty - \mu_t((-\infty, x)) M_1(dx) \\
    &= \int_0^\infty K(t,x) \,M_1(dx) = \int_0^t K(t,x) \,M_1(dx).
\end{align*}

We give some examples. Of course, the simplest example is Brownian
motion, where $K(t,x) = \Ind_{(0,t]}(x)$ and $\mu_t(dx) =
\delta_t(dx) - \delta_0(dx)$.  For $\alpha \geq 0$ and $0 \leq t < 1$
define
\[ 
X^{(\alpha)}_t = \int_0^t \left( \frac{1-t}{1-x} \right)^\alpha\,M_1(dx). 
\]
For $\alpha = 0$ we get the Brownian motion and for $\alpha = 1$ we
get the usual Brownian bridge. The integration kernel
$K(t,x)$ is such that, for $0<x<t$,
\[
K(t,x) = (1-t)^\alpha (1-x)^{-\alpha}\, \Ind_{(0,t]}(x),\quad 
\frac{\partial}{\partial x} K(t,x) = \alpha
(1-t)^\alpha (1-x)^{-\alpha-1}.
\]
Hence the measures $(\mu_t)_{t \in [0,1]}$ become
\[ 
\mu_t(dx) = \delta_t(dx) - (1-t)^\alpha \delta_0(dx) - \alpha
\frac{(1-t)^\alpha}{(1-x)^{1+\alpha}}\, \Ind_{(0,t]}(x)\, dx. 
\]
A further example is the centered Ornstein-Uhlenbeck process with
stability parameter $\alpha > 0$ and diffusion parameter $\sigma > 0$,
given by
\[ 
X^{(\alpha, \sigma)}_t = \int_0^t \sigma e^{\alpha(x-t)} \, dW_x.
\]
The kernel is $K(t,x) = \sigma e^{\alpha(x-t)}\, \Ind_{(0,t]}(x)$ and
thus $\frac{\partial}{\partial x} K(t,x) =
\alpha \sigma e^{\alpha (x-t)}$, $0 < x < t$. Therefore
\[ 
\mu_t(dx) = \sigma \delta_t(dx) - \sigma e^{-\alpha t} \delta_0(dx) -
\alpha \sigma e^{\alpha (x-t)} \, \Ind_{(0,t]}(x)\,dx. 
\]

\subsection*{Fractional bridges}

For $0<H<1$, let $B_H$ be a standard fractional Brownian motion on the
real line and let $a_t$ be the function on the unit interval defined by 
\[
a^H_t=\frac{1+t^{2H}-(1-t)^{2H}}{2},\quad 0\le t\le 1.
\]
It is known that the fractional Brownian bridge process obtained by
pinning  $B_H$ to zero at time $t=1$ is equal in distribution to 
$Y_t=B_H(t)-a^H_t B_H(1)$,  see~\cite{Gas04}, \cite{Gor14}.

To obtain the fractional Brownian bridge in this work we apply Theorem
\ref{thmsingularshot} i) for $d=1$ and $0<\beta<1$ to get 
\[
Y_t\stackrel{d}{=} \mathrm{const} X(\delta_t-\delta_0-a^H_t(\delta_1-\delta_0)),\quad 0<H<1/2.
\]
Similarly, by Theorem \ref{thmsingularshot} ii) for $d=1$,
$1<\beta<2$, and $1/2<H<1$,
\[
Y_t\stackrel{d}{=} \mathrm{const} \int_\R 
\Big((t-x)_+^{H-1/2} -(1-a_t^H)(-x)_+^{H-1/2}-a_t^H(1-x)_+^{H-1/2}\Big)\,M_1(dx).
\]

%%%%%%%%%%%%%%%%%%%%%%%%%%%%%%%%%%%%%%%%

\subsection*{Membranes by soft boundary thinning}

In this subsection we discuss briefly an approach of extending
the method for extracting Gaussian bridge processes in one dimension to
a method for extracting Gaussian membranes in higher dimensions. 

We start again with the representation $B_t = M_1((-\Delta)^{-1/2}
\mu_t)$ in (\ref{brownianbridge}) of the Brownian bridge on $[0,1]$.
Here $\mu_t = \delta_t - \omega_t$, where the measure $\omega_t =
(1-t)\delta_0 - t \delta_1$ is the harmonic measure on the set
$\{0,1\}$ (the start and end point of the bridge) of a Brownian motion
starting in $t$. This observation leads us to defining a class of
Gaussian membranes as follows. Let $D \subset \R^d$ be a bounded
domain satisfying the Poincar\'e cone condition (see Definition~3.10
in~\cite{Mor10}). For $t \in D$, let $\omega_t$ denote the harmonic measure
\[ 
\omega_t(A) = \Prob(W(\tau) \in A), \qquad A \in \Bc(\R^d), 
\]
where $\{W(x),x\in\R^d\}$ is $d$-dimensional Brownian motion with $W(0)=t$ and
$\tau$ is the stopping time $\tau = \inf\{ s \geq 0: W(s) \in \partial D \}$, 
and let $\mu_t$ in this subsection denote the signed measure 
\[
\mu_t = \delta_t - \omega_t,\quad t\in D.
\]

Given a continuous function $f: \partial D \rightarrow \R$, a solution to the Dirichlet problem with boundary value $f$ is a function $u: \bar D \rightarrow \R$ which is harmonic in $D$ and satisfies $u(t) = f(t)$ for $t \in \partial D$. It can be shown that the unique solution is
\equ[E:DIRI]{ u(t) = \int_{\R^d} f(y) \omega_t(dy), \qquad t \in \bar D,}
see \cite{Mor10}, Corollary~3.40. 

We have $\int_{\R^d} \mu_t(dy)=0$ and, since the function $u(t)=t$ is harmonic on
$\R^d$,  
\[
\int_{\R^d} y\,\mu_t(dy)=t-\int_{\partial D} y\,\omega_t(dy)=0.
\]
Thus, $\mu_t\in\Mc_2$.  Considering now the Gaussian selfsimilar
random fields $\mu\mapsto X(\mu)$ in Theorem \ref{thmsingularshot} i),
or Theorem \ref{thmshotnoise} i), with self-similarity index $H\in
(0,1/2)$, we can introduce for $d\ge 1$ a collection $(X_t)_{t \in
  \bar D}$ of zero mean Gaussian random variables defined by $t\mapsto
X_t=X(\mu_t)$, with finite covariance
\begin{equation}\label{covsoft}
\E X_s X_t = \mathrm{const}
\int_{\R^d \times \R^d} |y - y'|^{2H} \mu_s(dy)\mu_t(dy'),\quad s,t\in
D. 
\end{equation}
To construct a Brownian membrane on $D$ vanishing on $\partial D$, we 
apply Theorem \ref{thmsingularshot} ii) with $H=1/2$ and put
\[
X_t=\int_{\R^d}\int_{\R^d}
\frac{\mu_t(dy)}{|y-x|^{(d-1)/2}}\,M_d(dx),\quad t\in \R^d.
\]
The variance now is
\begin{equation}\label{covsoftbrownian}
\E X_t^2=\int_{\R^d} \Big(\frac{1}{|t-x|^{(d-1)/2}}
-\int_{\partial D} \frac{\omega_t(dy)}{|y-x|^{(d-1)/2}}\Big)^2\,dx,
\end{equation}
assuming the integral exists.  Here, we restrict to $H=1/2$ but the
same construction works for $H<1$.

The following Proposition justifies the term membrane for this
class of processes in the domain $D$ which vanishes on the boundary
$\partial D$. 

\begin{proposition} The processes $(X_t)_{t \in \bar D}$ described above for $0<H\le
  1/2$ are well-defined, and 
  for $x \in \partial D$ we have $\lim_{t \rightarrow x} \E X_t^2 = 0$.
\end{proposition}

\begin{proof}
First, we show that, for $x \in \partial D$, we have $\lim_{t
  \rightarrow x} \omega_t = \delta_x$ in the weak sense: Assume that
$f: \R^d \rightarrow \R$ is continuous and bounded. By~\eqref{E:DIRI},
the function 
\[ 
u(t) = \int_{\R^d} f(y)\, \omega_t(dy),
\]
is continuous on $\bar D$ with $u(t) = f(t)$ on $\partial D$, and thus
\[ 
\lim_{t \rightarrow x} \int_{\R^d} f(y)\, \omega_t(dy) = \lim_{t \rightarrow x} u(t) = u(x) = f(x) = \int_{\R^d} f(y)\, \delta_x(dy). 
\]

 Next, for $0<H<1/2$ by (\ref{covsoft}), 
\[
    \E X_t^2 
=\mathrm{const}\Big( \int_{\R^d \times \R^d} |y - y'|^{2H} \omega_t(dy)\omega_t(dy')
- 2 \int_{\R^d} |t - y|^{2H} \omega_t(dy)\Big).
\]
From the first part of the proof it follows that $\omega_t \otimes
\omega_t \rightarrow \delta_x \otimes \delta_x$ as $t \rightarrow x$
and thus
\begin{align*}
    &\lim_{t \rightarrow x} \int_{\R^d \times \R^d} |y - y'|^{2H}
  \omega_t(dy) \omega_t(dy')\\
&\qquad \qquad = \lim_{t \rightarrow x} \int_{\R^d \times \R^d} |y -  y'|^{2H} 
\delta_x(dy) \delta_x(dy') = 0.
\end{align*}
Moreover, we have $\delta_t \otimes \omega_t \rightarrow \delta_x
\otimes \delta_x$ as $t \rightarrow x$. Hence,
  \begin{align*}
    \lim_{t \rightarrow x} \int_{\R^d} |t - y|^{2H} \omega_t(dy) &
= \lim_{t \rightarrow x} \int_{\R^d \times \R^d} |y - y'|^{2H} \omega_t(dy) \delta_t(dy') \\
      &= \int_{\R^d \times \R^d} |y - y'|^{2H} \delta_x(dy)
    \delta_x(dy') = 0.
  \end{align*}
Turning to the case $H=1/2$, by rewriting (\ref{covsoftbrownian}),
\[
\E X_t^2=\int_{\partial D\times\partial D} F_t(y,y')\,\omega_t(dy)\omega_t(dy'),
\]
where, using the short notation $\delta=(d-1)/2$,
\[
F_t(y,y')=\int_{\R^d}
\Big(\frac{1}{|z|^\delta}-\frac{1}{|z+y-t|^\delta}
   -\frac{1}{|z+t-y'|^\delta}+\frac{1}{|z+y-y'|^\delta}\Big)
\frac{dz}{|z|^\delta}.
\]
For any $t\in \R^d$, the function $F_t$ is bounded on $\partial
D\times\partial D$. Hence, 
\[
\E X^2_t \to F_x(x,x)=0, \quad t\to x\in\partial D. 
\qedhere
\]
\end{proof}

%The harmonic measure on the unit ball in $\R^d$ is given by 
%the Poisson formula
%\[
%\omega_t(B)=\int_B \frac{1-|t|^2}{|t-y|^d}\,\sigma(dy).
%\] 

%%%%%%%%%%%%%%%%%%%%%%%%%%%%%%%%%%%%%%%%%%%%%%%%%%%%%%%%%%%%%%%%%%%%%%%%%%%%%%%%%%%%%%%%%%%%%%%%%%%%%%%%%%%%%%%%%

\section{Gaussian membranes by hard boundary thinning}

In this section we consider another way to construct membranes on
domains in $\R^d$. Again, let $D \subset \R^d$ be a bounded domain,
and let $\beta < d$ be a real number. We modify the
basic model~\eqref{defrandomballs} and consider 
\[
X(\mu)=\int_{\R^d\times \R_+} \mu(B(x,u)) M^D_\beta(dx, du),
\]
where $M^D$ is the Gaussian random measure on $\R^d \times \R_+$ with
control measure \linebreak $\nu^D(dx,du) = u^{-\beta-1} \Ind(B(x,u) \subset
D)\,dxdu$. Hence in the random balls interpretation, the intensity
measure is modified such that balls which do not fall entirely inside
the domain $D$ are discarded.  This model is well-defined for any
$\mu\in \Mc$ and $\beta<d$.  Here, we will apply the extraction
principle $t\mapsto X(\delta_t)$ and consider 
\equ[E:2]{ 
W_\beta(t) =
  \int_{\R^d\times \R_+} \delta_t(B(x,u)) M^D_\beta(dx, du), \quad
  t\in \R^d.
} 
Since $D$ is bounded there is an $N > 0$ such that $\Id(B(x,u)
\subset D) = 0$ for all $u > N$ and $x \in D$. Hence, 
$W_\beta$ is a Gaussian random field with finite covariance function
\begin{align} \label{covdomain}
\E W_\beta(s) W_\beta(t) &= \int_{\R^d\times \R_+} \Id(s,t \in B(x,u)
\subset D) u^{-\beta-1} \, dx du\\
&\le \int_0^N |B(s,u)\cap D|   u^{-\beta-1}\, du<\infty,\nonumber
\end{align}
for all $s,t \in \R^d$.
In particular, $W_\beta(s)$ and $W_\beta(t)$ are
independent if and only if there is no $x \in \R^d$ and $u>0$
such that $B(x,u)$ is a subset of $D$ and covers both points $s \in
\R^d$ and $t \in \R^d$.  By the Lebesgue dominated convergence
theorem, $\E(W_\beta(t)^2)\to 0$ as $t\to t_0\in \partial  D$, which
again justifies the notion of a membrane. 

The random fields $W_\beta$ are not selfsimilar in the sense that
$W(cs)\stackrel{d}{=} c^H W(s)$ for some $H$. But we will see that
they are selfsimilar in the following local sense (see \cite{Fal02}
for more details and background): A zero mean random field
$W=(W(t))_{t \in E}$, $E \subset \R^d$ open, is said to be
locally asymptotically selfsimilar with index $H$ in the point
$z \in E$, if $H$ is the supremum of all $\gamma \geq 0$ such that
\equ[E:30]{ \eps^{-\gamma} ( W(z + \eps s) - W(z) ) \longrightarrow 0
} as $\eps \rightarrow 0$ in the sense of finite dimensional
distributions. Then the random field $T^z=(T^z(s))_{s \in \R^d}$ with
\equ[E:31]{ T^z(s) = \lim_{\eps \rightarrow 0} \tau(\eps) ( W(z + \eps
  s) - W(z) ) } is called the tangent field at $z \in \R^d$, if
$\tau$ is a suitable scaling function such that the limit exists in
the sense of finite dimensional distributions and $T^z \not\equiv
0$. The tangent field is selfsimilar with index $H$ and uniquely
determined modulo constants. By ~(\ref{E:30}), $\tau(\eps) \eps^\gamma
\rightarrow \infty$ as $\eps \rightarrow 0$ for all $\gamma < H$ and
$\tau(\eps) \eps^\gamma \rightarrow 0$ as $\eps \rightarrow 0$ for all
$\gamma > H$.  It is not necessarily the case, however, that
$\tau(\eps) \sim c \eps^{-H}$, some $c > 0$.

\begin{theorem}\label{T:LASS}
  The Gaussian membrane $W_\beta$ is in every point $z \in D$
  locally asymptotically selfsimilar with index 
\[
H = 
\begin{cases}
(d-\beta)/2,&   d-1<\beta<d, \\
1/2,& \beta \leq d-1.
\end{cases}
\]
Moreover, the tangent field $T^z$ in $z$ is a multiple of fractional
Brownian motion with Hurst index $H$.  The scaling function $\tau$ is
given by 
\equ[E:33]{ \tau(\eps) = \begin{cases} \eps^{(\beta-d)/2},
    &\text{$d-1 < \beta < d$,} \\ ( - \eps \ln(\eps)
    )^{-1/2}, &\text{$\beta = d-1$,} \\ \eps^{-1/2},
    &\text{$\beta < d-1$.} \end{cases} }
\end{theorem}

\begin{lemma}\label{lemma}
  Let $\tau(\eps)$ and $H$ be as in Theorem~\ref{T:LASS}. For $M \geq
  0$ and $t\in\R^d$, define
\equ[E:PsiTilde]{ 
\widetilde \Psi(M,t) = \lim_{\eps \rightarrow 0} 
\tau(\eps)^2 \int_{\R^d} \int_0^M |\delta_{\eps
      t}(B(x,u)) - \delta_0(B(x,u))| u^{-\beta-1} \,du dx. } 
Then
  $\widetilde \Psi(M, t) = c \ |t|^{2H}$, where $0 \leq c < \infty$
  is a constant depending on $M$ and $\beta$ but independent of $t$.
\end{lemma}

\begin{proof}
By a change of the order of integration in~\eqref{E:PsiTilde} we obtain
\begin{align*}
\widetilde \Psi(M,t)
&= \lim_{\eps \rightarrow 0} \tau(\eps)^2 \int_0^M u^{-\beta-1} 
|B(\eps t, u) \setDiff B(0,u)| \,du \\
&= \lim_{\eps \rightarrow 0} \tau(\eps)^2 \int_0^M u^{d-\beta-1}
|B(\eps t/u,1) \setDiff B(0, 1)| \,du.
  \end{align*}
Hence,
\[ 
\widetilde \Psi(M,t) = \lim_{\eps \rightarrow 0} \tau(\eps)^2\eps^{d-\beta} 
|t|^{d-\beta} \int_0^{\frac{M}{\eps |t|}} u^{d-\beta-1} |B(e / u, 1)) \setDiff B(0, 1)| \,du 
\]
for some $e \in \S^{d-1}$.  Using the function 
$V(u)$ in (\ref{defvolfunction}) for the volume of the intersection of
two balls of radius $1$ and center distance $u$,  
\[ 
|B(e/ u, 1) \setDiff B(0, 1)| = \begin{cases}
2 V(0),&    u \leq 1/2,\\ 
2 (V(0) - V(1/u)), & u > 1/2.
\end{cases}
\]
By L'Hospital's rule
\al{
  \lim_{u \rightarrow \infty} u \,|B(e / u, 1) \setDiff
    B(0, 1)| 
= \lim_{u \rightarrow \infty} 4 \nu_{d-1}u \int_0^{1/(2u)} (1-s^2)^{\frac{d-1}{2}}\,ds 
      = 2 v_{d-1}
  }
and so
\[ 
\widetilde \Psi(M,t) = \mathrm{const} 
\lim_{\eps \rightarrow 0} \tau(\eps)^2 \eps^{d-\beta} |t|^{d-\beta}
\int_0^{M/(\eps |t|)} u^{d-\beta-1} \min\{1, u^{-1} \}\, du. 
\]
By  evaluating this integral expression separately for the three
different intervals of $\beta$ and the corresponding scaling functions
$\tau(\eps)$, we obtain 
\[ 
\widetilde \Psi(M,t) = \mathrm{const}\ |t|^{2H}
\]
for the choice the of Hurst index stated in Theorem \ref{T:LASS}. 
\end{proof}

\begin{proof}[Proof of Theorem~\ref{T:LASS}]
  We note that the limit of a sequence of Gaussian processes is
  Gaussian and that Gaussian processes are determined by their
  two-dimensional distributions. Fix an element $z \in D$. We define
  formally
  \begin{align}
    T'(t) &= \lim_{\eps \rightarrow 0} \tau(\eps) 
( W_\beta(z + \eps t) - W_\beta(z) ) \label{E:34} \\
      &= \lim_{\eps \rightarrow 0} \int_{\R^d\times\R_+} 
\tau(\eps) (\delta_{z + \eps t} - \delta_z)(B(x,u)) \,M_\beta(dx, du). \notag
  \end{align}
Then the covariance of $T'$ is given by
  \al{ 
&\E T'(s) T'(t) = \lim_{\eps \rightarrow 0} \tau(\eps)^2 
\int_{\R^d\times \R_+} (\delta_{z+\eps s} - \delta_z)(\delta_{z + \eps t} -
    \delta_z)(B(x,u))\, \nu^D(dx,du). 
}
We have
  \al{ 
(\delta_{z+\eps s} - \delta_z)&(\delta_{z+\eps t} - \delta_z) =
%    \frac{1}{2} \Big( 2 \delta_{z+\eps s} \delta_{z+\eps t} - 2
%    \delta_{z+\eps s} \delta_z - 2 \delta_{z+\eps t} \delta_z +
%    2\delta_z^2 \\
%    &\qquad \qquad + \delta_{z+\eps s}^2 - \delta_{z+\eps s}^2 +
% \delta_{z+\eps t}^2 - \delta_{z+\eps t}^2 \Big) \\ 
\frac{1}{2} (
    (\delta_{z+\eps s} - \delta_z)^2 + (\delta_{z+\eps t} -
    \delta_z)^2 - (\delta_{z+\eps s} - \delta_{z+\eps t})^2 ). 
}
Moreover, 
\[
(\delta_{z+\eps s} - \delta_{z+\eps t})(B(x,u)) =
  (\delta_{z + \eps s - \eps t} - \delta_z)(B(x - \eps t, u)),
\]
and thus
\al{ & \int_{\R^d\times\R_+} (\delta_{z+\eps s} - \delta_{z+\eps t})^2(B(x,u))
    u^{-\beta-1} \Id(B(x,u) \subset D) \,du dx \\ 
&\qquad \qquad = \int_{\R^d\times\R_+} (\delta_{z + \eps (s - t)} - \delta_z)^2(B(x, u))
    u^{-\beta-1} \Id(B(x,u) \subset D- \eps t)\, dx du. 
} 
We obtain
  \equ[E:35] { \E T'(s) T'(t) = \frac{1}{2} \left( \Psi(s,D) +
    \Psi(t,D) - \Psi(s-t , D-\eps t) \right) } 
with
\begin{align*} 
    \Psi(t,D) &= \lim_{\eps \rightarrow 0} \tau(\eps)^2
    \int_{\R^d\times\R_+} (\delta_{z+\eps t} - \delta_z)^2(B(x,u)) 
u^{-\beta-1} \Id(B(x,u) \subset D) \,du dx \\
  &= \lim_{\eps \rightarrow 0} \tau(\eps)^2 \int_{\R^d\times\R_+} 
(\delta_{\eps t} - \delta_0)^2(B(x,u)) u^{-\beta-1} \Id(B(x,u) \subset D - z)\ dudx.
  \end{align*}
We evaluate $\Psi(t,D)$. The third term in (\ref{E:35}) then requires
only small modifications because we have to work with $D-\eps t$
instead of $D$. Recalling the definition of $\widetilde \Psi(M,t)$
in~\eqref{E:PsiTilde}, $\widetilde \Psi(\cdot,t)$ is a continuous,
monotone increasing function with $\widetilde \Psi(0,t) = 0$.  Since
$D$ is bounded there is an $N > 0$ which only depends on $D$ such that
$\Id(B(x, u) \subset D) = 0$ for all $u > N$ and $x \in \R^d$. Hence,
$\Psi(t,D) \leq \widetilde \Psi(N,t)$.  A careful reading of the proof
of Lemma \ref{lemma} shows that, for any $v \in \R$, we may replace
$M$ by $M + \eps v$ in the right hand side of~\eqref{E:PsiTilde}
without changing the constant in $\widetilde \Psi(M, t) =
\mathrm{const} \ |t|^{2H}$.  Therefore we can find an $M$, $0 \leq M
\leq N$, depending on $D$, $\beta$, and $z$, but independent of $t$,
such that $\widetilde \Psi(M,t) = \Psi(t,D)$. Hence by Lemma
\ref{lemma}, $\Psi(t,D) = c \ |t|^{2H}$, with a constant $c$
independent of $t$.  It follows immediately that $\Psi(s,D) = c \,
|s|^{2H}$ and $\Psi(s-t,D-\eps t) = c \,|s-t|^{2H}$, and hence
  \[ 
\E T'(s) T'(t) = c \ ( |s|^{2H} + |t|^{2H} - |s-t|^{2H} ). \qedhere 
\]
\end{proof}

%%%%%%%%%%%%%%%%%%%%%%%%%%%%%%%%%%%%%%%%%%%%%%%%%%%%%%%%%%%%%%%%%%%%%%%%%%%%%%%%%%%%%%%%%%%%%%%%%%%%%%%%%%%%%%%%%

\subsection*{The hard boundary thinning bridge on $[0,T]$}

We consider the Gaussian membrane $W_\beta$ for the special case
$d=1$ and $D = (0,T)$, some $T > 0$.  By~\eqref{covdomain},
\begin{align*} 
\E W_\beta(s) W_\beta(t)&=\int_0^T \int_0^\infty
\Id(0<x-u<s,t<x+u<T)\, u^{-\beta-1} du dx\\
&=f_\beta(s \vee t) + f_\beta(T-s \wedge t) - f_\beta(|s - t|) -
f_\beta(T), 
\end{align*}
where 
\[
f_\beta(x) =
\begin{cases} 
\frac{2^\beta}{\beta(1-\beta)} x^{1-\beta},& \beta<d, \beta\not=0,\\
- x \ln x, & \beta=0.
\end{cases}
\]
We point out that the case $\beta = -1$ is the classical Brownian bridge on $[0,T]$ with covariance function
\[ 
\E W_{-1}(s) W_{-1}(t) = \frac{1}{2} \,s \wedge t\,(T- s\vee t). 
\]
For any other value of the parameter $\beta$, however, the hard
boundary Gaussian bridge is different from the fractional Brownian
bridge on $[0,T]$ obtained as fractional Brownian
  motion pinned to zero at time $T$.
%\begin{remark}
%  The supremum of the variance of $W_\beta$ is attained at $T/2$ and is given by
%  \[ \E W_\beta(T/2)^2 = \begin{cases} T \ln(2), &\text{in the case $\beta = 0$,} \\ \frac{2^\beta T^{1-\beta}}{\beta (1- \beta)} \left( 2^\beta - 1 \right), &\text{in the case $\beta \neq 0$.} \end{cases} \]
%\end{remark}
To conclude we provide an additional result on the relation between the hard
boundary bridge and fractional Brownian motion.

\begin{proposition}
  Let $B=(B_t)_{t \in [0,T]}$ be standard linear Brownian motion independent from $W_\beta$ and define the Gaussian martingale $Y_\beta=(Y_\beta(t))_{t \in [0,T]}$ by
  \equ[E:Y]{ 
Y_\beta(t) = \sqrt{\frac{2^\beta}{\beta}} \int_0^t \sqrt{x^{-\beta} +
  (T-x)^{-\beta}} dB_x. 
}
  Then, for $0 < \beta < 1$, $W_\beta + Y_\beta$ is (up to constant) a fractional Brownian motion with Hurst index $H = (1-\beta)/2$.
\end{proposition}

\begin{proof}
  Put $Z_\beta = W_\beta + Y_\beta$. As the sum of two independent
  Gaussian processes, $Z_\beta$ is a Gaussian process as well.  Thus, it
  is enough to show that the covariance $\E
  Z_\beta(s) Z_\beta(t)$ is given by a multiple of $C_\beta(s,t)$, with
  \[ 
C_\beta(s,t) = s^{1-\beta} + t^{1-\beta} - |s-t|^{1-\beta}. 
\]
It is easily seen that
\[ 
\E W_\beta(s) W_\beta(t) = \frac{2^\beta}{\beta(1-\beta)} (
C_\beta(s,t) - C_\beta(s \wedge t,T) ). 
\]
  Since $\E Z_\beta(s) Z_\beta(t) = \E W_\beta(s) W_\beta(t) + \E Y_\beta(s) Y_\beta(t)$ for all $s,t \in [0,T]$, it is enough to show that $\E Y_\beta(s) Y_\beta(t) = 2^\beta C_\beta(s \wedge t,T) / (\beta(1-\beta))$. In fact,
  \al{
    \E Y_\beta(s) Y_\beta(t) &= \frac{2^\beta}{\beta} \int_0^{s \wedge t} x^{-\beta} + (T-x)^{-\beta} dx \\
                             &= \frac{2^\beta}{\beta(1-\beta)} \left( (s \wedge t)^{1-\beta} - (T - s \wedge t)^{1-\beta} + T^{1-\beta} \right) \\
                             &= \frac{2^\beta}{\beta(1-\beta)} C_\beta(s \wedge t,T). \qedhere }
\end{proof}

We may extend the definition of $Y_\beta$ in~\eqref{E:Y} to $-1 <
\beta < 0$. Then $Y_\beta$ becomes a purely imaginary Gaussian
process. Setting $Z_\beta = W_\beta + Y_\beta$ as before yields a
complex (centered) Gaussian process with
  \[ \E Z_\beta(s) Z_\beta(t) = \E Z_\beta(s) \overline{Z_\beta(t)} = \frac{2^\beta}{\beta(1-\beta)} C_\beta(s, t). \]

%%%%%%%%%%%%%%%%%%%%%%%%%%%%%%%%%%%%%%%%
%%%%%%%%%%%%%%%%%%%%%%%%%%%%%%%%%%%%%%%%

\bibliographystyle{plain}
\bibliography{gaussianFieldsRefs}

%%%%%%%%%%%%%%%%%%%%%%%%%%%%%%%%%%%%%%%%
%%%%%%%%%%%%%%%%%%%%%%%%%%%%%%%%%%%%%%%%
%%%%%%%%%%%%%%%%%%%%%%%%%%%%%%%%%%%%%%%%

\end{document}